\documentclass[11pt,oneside,dvipsnames]{amsart}
\usepackage{stackrel}
\usepackage[french, english]{babel}
\usepackage{a4wide}
\usepackage[utf8,latin1]{inputenc}
\usepackage{times}
\usepackage{adjustbox}
\usepackage{float}
\usepackage[cyr]{aeguill}
\usepackage{amsmath,amscd}
\usepackage{amssymb}
\usepackage{mathrsfs}
\usepackage{amsthm}
\usepackage{geometry}
\usepackage{graphicx}
\usepackage{epstopdf}
\usepackage{amsfonts}
\usepackage{amsmath}
\usepackage{amsmath,mathtools}
\usepackage{amsmath,graphicx}
\usepackage{amscd}
\usepackage{dsfont}
\usepackage{latexsym}
\usepackage{verbatim}
\usepackage{comment}
\usepackage{tikz-cd}
\usetikzlibrary{matrix,arrows}
\usepackage{tikz}
\usetikzlibrary{tikzmark}
\usepackage{tensor}
\usepackage{leftidx}
\usepackage[retainorgcmds]{IEEEtrantools}
\usepackage[title]{appendix}
\usepackage[mathscr]{euscript}
\usepackage[colorinlistoftodos]{todonotes}
\usepackage{bbm}
\usepackage{url}
\usepackage{upgreek}
\usepackage{mathtools, nccmath}
\usepackage[utf8]{inputenc}
\usepackage{multirow}
\usepackage{rotating}
\usepackage{marginnote}
\usepackage{hyperref} 
\usepackage{nameref}
\usepackage{ytableau}
\usepackage{youngtab}

\usepackage[all]{xy}
\usepackage[T1]{fontenc}

\usepackage{hyperref}
\usepackage{array}
\usepackage{tabularx}
\usepackage{yfonts}
\usepackage{setspace}  
 \usepackage[foot]{amsaddr}
\makeatletter
\renewcommand\subsubsection{\@startsection{subsubsection}{3}{\z@}%
                                     {-2ex\@plus -1ex \@minus -.2ex}%
                                     {-0.5em}
                                     {\normalfont\normalsize\bfseries}}
                                     
                                     \renewcommand\subsection{\@startsection{subsection}{3}{\z@}%
                                     {-3.25ex\@plus -1ex \@minus -.2ex}%
                                     {-0.5em}
                                     {\normalfont\normalsize\bfseries}}

\makeatother

\theoremstyle{plain}
\newtheorem{Thm}{Theorem}
\newtheorem*{Thm*}{Theorem}
\newtheorem{ThmA}{Theorem}

\newtheorem{Prop}[Thm]{Proposition}
\newtheorem{Lem}[Thm]{Lemma}

\newtheorem*{Con*}{Conjecture}

\theoremstyle{definition}

\newtheorem{Rem}[Thm]{Remark}

\newtheorem{Defn}[Thm]{Definition}

\newtheoremstyle{named}{}{}{}{}{\bfseries}{.}{.5em}{\thmnote{#3}#1}
\theoremstyle{named}

\DeclareMathOperator{\gl}{\mathfrak{gl}}
\DeclareMathOperator{\GL}{GL}

\DeclareMathOperator{\Hom}{Hom}

\DeclareMathOperator{\Id}{Id}

\DeclareMathOperator{\lisom}{\!\!\smash{\begin{array}{c}\sim\\[-1em]
\longrightarrow\end{array}}\!\!}

\DeclareMathOperator{\rk}{rk}

\DeclareMathOperator{\Rep}{Rep}

\DeclareMathOperator{\ds}{\!/\mkern-2mu/\mkern-2mu}

\makeatletter
\newcommand*{\rom}[1]{\expandafter\@slowromancap\romannumeral #1@}
\makeatother

\renewcommand{\descriptionlabel}[1]{\hspace\labelsep\upshape\bfseries #1.}
\makeatletter
\let\orgdescriptionlabel\descriptionlabel
\renewcommand*{\descriptionlabel}[1]{%
  \let\orglabel\label
  \let\label\@gobble
  \phantomsection
  \edef\@currentlabel{#1}%
  \let\label\orglabel
  \orgdescriptionlabel{#1}%
}
\makeatother

\usepackage{subfig}	 

\title{The tame Deligne-Simpson problem}
\author{Cheng Shu}
\address[C. Shu]{Institute for Theoretical Sciences, Westlake University, Hangzhou, China}
\email[C. Shu]{shuc@zju.edu.cn}

\usepackage{pxfonts}
\usepackage{indentfirst}
\usepackage{setspace}  
\usepackage{hyperref}
\colorlet{ivory}{Apricot!30!}
\colorlet{space}{black!85!}
\definecolor{bgc}{RGB}{29, 44, 46}
\definecolor{txt}{RGB}{223, 222, 189}
\definecolor{cmd}{RGB}{206, 151, 88}
\begin{document}
\let\bs\boldsymbol
\begin{abstract}
The objective of this article is to prove the necessity statement in Crawley-Boevey's conjectural solution to the (tame) Deligne-Simpson problem. We use the nonabelian Hodge correspondence, variation of parabolic weights and results of Schedler-Tirelli to reduce to simpler situations, where every conjugacy class is semi-simple and the underlying quiver is (1) an affine Dynkin diagram or (2) an affine Dynkin diagram with an extra vertex. In case (1), a nonexistence result of Kostov applies. In case (2), the key step is to show that simple representations, if exist, lie in the same connected component as direct sums of lower dimensional ones.
\end{abstract}
\maketitle

\tableofcontents
\addtocontents{toc}{\protect\setcounter{tocdepth}{-1}}
\setcounter{tocdepth}{1}
\numberwithin{Thm}{section}
\numberwithin{equation}{section}
\addtocontents{toc}{\protect\setcounter{tocdepth}{1}}
\section{Introduction}\label{INT}

For a given tuple of conjugacy classes $(C_j)_{1\le j\le k}$ in $\GL(\mathbb{C}^n)$, does there exist matrices $A_j\in C_j$ such that the following conditions hold?
\begin{itemize}
\item $A_1\cdots A_k=\Id$, and 
\item there is no nontrivial proper vector subspace of $\mathbb{C}^n$ that is preserved by every $A_j$; in this case, the tuple $(A_j)_j$ is called \textit{irreducible}.
\end{itemize}
This question was posed by Deligne, and the first attempt was made by Simpson \cite{Si91}; assuming one of the conjugacy classes to be generic regular semi-simple, Simpson obtained a necessary and sufficient condition on the tuple $(C_j)_j$ for this question to have an affirmative answer. Some earlier attempts were also made by Kostov, and the problem was since known as the Deligne-Simpson problem; see \cite{Kos99}, \cite{Kos01} as well as a survey \cite{Kos04}.

Despite its linear-algebra look, the problem is most natural in its geometric form, and the underlying geometric objects are known as  character varieties. For any genus $g$ and any tuple of closures of conjugacy classes $\overline{\mathcal{C}}=(\bar{C}_j)_{1\le j\le k}$ of $\GL_n$, the associated character variety is the affine GIT quotient
\begin{equation}
\mathcal{M}_g(\overline{\mathcal{C}}):=\{((A_i,B_i)_{1\le i\le g},(X_j)_{1\le j\le k})\in \GL_n^{2g}\times\prod_{j=1}^k\bar{C}_j\mid\prod_{i=1}^g[A_i,B_i]\prod_{j=1}^kX_j=\Id\}\ds\GL_n,
\end{equation}
where the bracket means commutator $[A_i,B_i]=A_iB_iA_i^{-1}B_i^{-1}$. As is clear from the definition, only the case of $g=0$ is relevant to the Deligne-Simpson problem, and we will omit the subscript $g$ if it is clear from the context. This variety parametrises semi-simple representations of the fundamental group of a punctured Riemann surface, or equivalently, semi-simple local systems with prescribed monodromies. The Deligne-Simpson problem amounts to asking whether there exist irreducible local systems with monodromies at the punctures lying in $(C_j)_{1\le j\le k}$. This geometric aspect makes it possible to employ tools like nonabelian Hodge theory (see \cite{Si94a} and \cite{Si94b}) to attack the problem. Indeed, in Simpson's original article \cite{Si91}, solutions were found by going back and forth along the nonabelian Hodge correspondence.

However, geometry lacks the appropriate language to organise the combinatorial information in the monodromies in a meaningful way so as to formulate a clean answer to the problem. In \cite{CB04}, Crawley-Boevey reformulated the problem in terms of the Kac-Moody root system of certain star-shaped graphs associated the conjugacy classes $(C_j)_j$, and a conjectural necessary and sufficient condition was proposed. This Kac-Moody picture much clarified the problem and indeed allowed him to make significant progress towards the answer. The sufficiency statement in this conjecture was lated confirmed in his joint work with Shaw \cite{CBS}, where multiplicative preprojective algebras were introduced in connection with Katz's middle convolution operation, which proves to be a useful tool for many purposes and is interesting in its own right. The purpose of our article is to prove the necessity statement in Crawley-Boevey's conjecture, thus giving a definite answer to the Deligne-Simpson problem. As we will see, our solution to Crawley-Boevey's conjecture brings in nonabelian Hodge theory again and fully exploits the flexibility that it provides.

In the rest of this introduction, we will recall in more detail Crawley-Boevey's formulation of the Deligne-Simpson problem, multiplicative quiver varieties and their relation with character varieties, followed by the statement of our main theorem.

\subsection{Multiplicative quiver varieties}\label{subset-MQV}\hfill

It is a classical result of Kraft-Procesi \cite{KP} that closures of adjoint orbits in $\mathfrak{gl}_n$ can be identified with quiver varieties of type A. This construction is applied to a tuple of closures of conjugacy classes of $\GL_n$ in the work of Crawley-Boevey-Shaw \cite{CBS}, where they give an identification between character varieties for $\mathbb{P}^1$ and multiplicative quiver varieties for star-shaped quivers. 

A quiver $Q=(Q_0,Q_1)$ consists of a vertex set $Q_0$ and an arrow set $Q_1$. We denote by $h$ and $t$ the two maps from $Q_1$ to $Q_0$ which sends an arrow to its head and tail respectively. Let $Q_1^{\ast}$ be a set of arrows in bijection with $Q_1$, which contains for any arrow $a:v\rightarrow w$ in $Q_1$ between vertices $v$ and $w$ an arrow $a^{\ast}:w\rightarrow v$. Let $\bar{Q}$ be the quiver with vertex set $Q_0$ and arrow set $\bar{Q}_1:=Q_1\sqcup Q_1^{\ast}$. Let $\epsilon:\bar{Q}_1\rightarrow\{\pm1\}$ be the function that takes the positive value precisely on $Q_1$. For any $\mathbf{d}=(d_v)_{v\in Q_0}\in(\mathbb{Z}_{\ge0})^{Q_0}$, called the dimension vector, write
$$
\Rep(\bar{Q},\mathbf{d})=\bigoplus_{a\in\bar{Q}_1}\Hom(\mathbb{C}^{d_{t(a)}},\mathbb{C}^{d_{h(a)}});
$$ 
an element of this vector space will be denoted by $\bs\phi:=(\phi_a)_{a\in\bar{Q}_1}$ and will be called a $\mathbf{d}$-dimensional representations of $\bar{Q}$. A subrepresentation of $\bs\phi$ consists of a subspace $V_v\subset \mathbb{C}^{d_v}$ for every $v$ such that $\phi_a(V_{t(a)})\subset V_{h(a)}$ for every $a$. A representation is simple if there is no nontrivial proper subrepresentation. Denote by $\Rep^{\circ}(\bar{Q},\mathbf{d})$ the open subset consisting of those $\bs\phi$ satisfying $\det(\Id+\phi_a\phi_{a^{\ast}})\neq 0$ for any $a$; such a $\bs\phi$ will be called an invertible representation of $\bar{Q}$. There is an action of $G:=\prod_{v\in Q_0}\GL_{d_v}$ on $\Rep(\bar{Q},\mathbf{d})$ preserving the open subset $\Rep^{\circ}(\bar{Q},\mathbf{d})$; the action sends $(g_v)_v\in G$ and $\bs\phi$ to $(g_{h(a)}\phi_ag_{t(a)}^{-1})_{a\in\bar{Q}_1}$. Choose a total ordering $<$ on $\bar{Q}$. Define
\begingroup
\allowdisplaybreaks
\begin{align}\label{eq-quasi-Ham}
\mu:\Rep^{\circ}(\bar{Q},\mathbf{d})&\longrightarrow G\\
\nonumber
\bs\phi&\longmapsto (\prod_{\substack{a\in\bar{Q}\\h(a)=v}}(1+\phi_a\phi_{a^{\ast}})^{\epsilon(a)})_{v\in Q_0}.
\end{align}
\endgroup
For any deformation parameter $\mathbf{q}=(q_v)_{v\in Q_0}\in(\mathbb{C}^{\ast})^{Q_0}$, regarded as a tuple of scalar matrices in $\prod_v\GL_{d_v}$, the associated multiplicative quiver variety is defined as the affine GIT quotient
\begin{equation}\label{eq-defn-multi-q.v.}
\mathcal{M}(\mathbf{q},\mathbf{d}):=\mu^{-1}(\mathbf{q})\ds G.
\end{equation}
A necessary condition for this variety to be nonempty is 
\begin{equation}\label{eq-qd=1}
\mathbf{q}^{\mathbf{d}}:=\prod_{v\in Q_0}q_v^{d_v}=1.
\end{equation}

Let $\bs\theta=(\theta_v)_{v\in Q_0}\in\mathbb{R}^{Q_0}$, which satisfies 
$$
\bs\theta\cdot\mathbf{d}:=\sum_{v\in Q_0}\theta_vd_v=0.
$$
We say that $\bs\phi\in\Rep(\bar{Q},\mathbf{d})$ is a $\bs\theta$-stable (resp. semi-stable) representation if for any nontrivial proper subrepresentation $(V_v)_v$, we have
$$
\sum_{v\in Q_0}\theta_v\dim V_v<0\text{ (resp. $\le0)$)}.
$$
Denote the open subset of $\bs\theta$-semi-stable representations by $\Rep^{\bs\theta-ss}(\bar{Q},\mathbf{d})$. Then, the multiplicative quiver variety with stability condition $\bs\theta$ is defined as
$$
\mathcal{M}_{\bs\theta}(\mathbf{q},\mathbf{d}):=(\Rep^{\bs\theta-ss}(\bar{Q},\mathbf{d})\cap\mu^{-1}(\mathbf{q}))\ds G.
$$

\subsection{Quiver description of character varieties}\label{subsec-Qui-Char}\hfill

Let $n\in\mathbb{Z}_{>0}$ and let $\bar{C}$ be the closure of a conjugacy class $C\subset\GL_n$. We would like to produce a quiver together with parameters $\mathbf{q}$ and $\mathbf{d}$ from $C$. Suppose that $\nu+1$ is the degree of the minimal polynomial of an element $A\in C$, and let $(\xi_i)_{0\le i\le \nu}$ be a tuple of complex numbers such that $\prod_{i=0}^{\nu}(A-\xi_i)=0$. For $1\le i\le \nu$, define
$$
d_i:=\rk (A-\xi_0)(A-\xi_1)\cdots(A-\xi_{i-1}),
$$
and $d_0=n$. Write $\mathbf{d}=(d_i)_i$. Define a quiver $Q$ with $Q_0=\{0,1,2,\ldots,\nu\}$ and arrows $a_i:i\mapsto i+1$. Consider the space of invertible representations $\Rep^{\circ}(\bar{Q},\mathbf{d})$ and the map $\mu$ as in (\ref{eq-quasi-Ham}). Let $\mu_{>0}$ be the composition of $\mu$ and the projection $\prod_{i\in Q_0}\GL_{d_i}\rightarrow\prod_{i\in Q_0\setminus\{0\}}\GL_{d_i}=:G_{>1}$. Define $q_0=\xi_0$ and $q_i=\xi_i\xi_{i-1}^{-1}$ for $1\le i\le \nu$, and regard $\mathbf{q}:=(q_i)_{i\in Q_0}$ as a central element of $\prod_{i\in Q_0}\GL_{d_i}$. Consider the map
\begingroup
\allowdisplaybreaks
\begin{align*}
\mu_{>1}^{-1}(\mathbf{q})\hookrightarrow\Rep^{\circ}(\bar{Q},\mathbf{d})&\longrightarrow \GL_n\\
(\phi_{a_i},\phi_{a_i^{\ast}})_i&\longmapsto q_0(1+\phi_{a_1^{\ast}}\phi_{a_1}),
\end{align*}
\endgroup
Then, by \cite[Lemma 9.1]{CB03} and \cite[Lemma 9.3]{Boa15}, the above map induces an isomorphism $\mu_{>1}^{-1}(\mathbf{q})\ds G_{>1}\cong \bar{C}$.

The above construction can be applied to a tuple of closures of conjugacy classes, resulting in an identification between character varieties for $\mathbb{P}^1$ and multiplicative quiver varieties for star-shaped quivers. Suppose that we have a tuple $\bar{\mathcal{C}}=(\bar{C}_1,\ldots,\bar{C}_k)$ of closures of conjugacy classes of $\GL_n$. These data define for each $1\le j\le k$ a type $A_{\nu_j+1}$ quiver $Q^{(j)}$ with vertices $Q^{(j)}_0=\{[j,i]\}_{0\le i\le \nu_j}$ and arrows $[j,i]\rightarrow[j,i+1]$, as well as a dimension vector $\mathbf{d}^{(j)}$ and a deformation parameter $\mathbf{q}^{(j)}$. Form a star-shaped quiver $Q$ by taking the disjoint union of all $Q^{(j)}$ and identifying the vertices $\{[j,0]\}$ for all $j$; the identified vertices will be denoted by $\star$ in $Q_0$, but we may let some $[j,0]$ represent $\star$. The resulting quiver is star-shaped as drawn below:
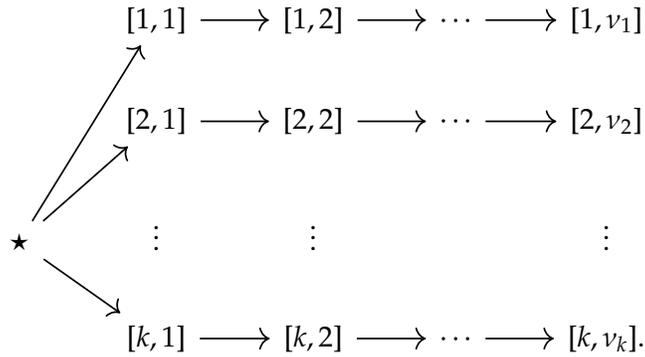
\[\begin{tikzcd}
	& {[1,1]} & {[1,2]} & \cdots & {[1,\nu_1]} \\
	& {[2,1]} & {[2,2]} & \cdots & {[2,\nu_2]} \\
	\star & \vdots & \vdots && \vdots \\
	& {[k,1]} & {[k,2]} & \cdots & {[k,\nu_k]}.
	\arrow[from=3-1, to=1-2]
	\arrow[from=1-2, to=1-3]
	\arrow[from=1-3, to=1-4]
	\arrow[from=1-4, to=1-5]
	\arrow[from=3-1, to=2-2]
	\arrow[from=2-2, to=2-3]
	\arrow[from=2-3, to=2-4]
	\arrow[from=2-4, to=2-5]
	\arrow[from=3-1, to=4-2]
	\arrow[from=4-2, to=4-3]
	\arrow[from=4-3, to=4-4]
	\arrow[from=4-4, to=4-5]
\end{tikzcd}\]
Define $\mathbf{d}$ by $\mathbf{d}|_{Q^{(j)}}=\mathbf{d}^{(j)}$ for all $j$, and define $\mathbf{q}$ by $q_{\star}=\prod_{j=1}^kq_0^{(j)}$ and $q_{[j,i]}=q^{(j)}_i$ for $i>0$. We have $\Rep^{\circ}(\bar{Q},\mathbf{d})=\prod_{j=1}^k\Rep^{\circ}(\bar{Q}^{(j)},\mathbf{d}^{(j)})$. Denote by $\mu_{>0}:\Rep^{\circ}(\bar{Q},\mathbf{d})\rightarrow\prod_{\{[j,i]|i>0\}}\GL_{d_{[j,i]}}$ the direct product of the maps $\mu_{>0}$ associated to $Q^{(j)}$ as above. Then, we have an affine GIT quotient by $\prod_{\{[j,i]|i>0\}}\GL_{d_{[j,i]}}$
$$
\mu^{-1}_{>0}(\mathbf{q})\longrightarrow \prod_{j=1}^k\bar{C}_j.
$$ 
Denote by $\mathbf{m}:\prod_{j=1}^k\bar{C}_j\rightarrow\GL_n$ the multiplication map. Then, the above quotient map restricts to closed subvarieties:
\begin{equation}\label{eq-multi=char-1}
\mu^{-1}(\mathbf{q})\longrightarrow\mathbf{m}^{-1}(1).
\end{equation}
Passing to the quotient by $\GL_n$, we obtain an isomorphism:
\begin{equation}\label{eq-multi=char-2}
\mathcal{M}(\mathbf{q},\mathbf{d})\lisom\mathcal{M}(\overline{\mathcal{C}}).
\end{equation}
According to Crawley-Boevey-Shaw (see \cite[Lemma 8.3]{CBS}), there is a simple representation in $\mathcal{M}(\mathbf{q},\mathbf{d})$ if and only if there is an irreducible local system in $\mathcal{M}(\mathcal{C})$. The necessary condition for nonemptiness (\ref{eq-qd=1}) is equivalent to
$$
\prod_{j=1}^k\det A_j=1,\text{ for $A_j\in C_j$, $1\le j\le k$}.
$$

\subsection{Crawley-Boevey's conjecture}\hfill

We are almost ready to state Crawley-Boevey's conjectural solution to the Deligne-Simpson problem. A couple of notations and definitions are in order. Let $Q$ be a star-shaped quiver. For any vertex $v\in Q_0$, we denote by $e_v$ the corresponding coordinate vector in $\mathbb{Z}^{Q_0}$, and we will call $e_v$ a simple root. For any $\mathbf{d}\in\mathbb{Z}_{\ge0}^{Q_0}$, the support of $\mathbf{d}$ is the subquiver obtained by removing vertices $v$ with $d_v= 0$ and edges connecting to such vertices. Denote by $(-,-)$ the symmetric bilinear form on $\mathbb{Z}^{Q_0}$ defined by
$$
(\mathbf{d}^{(1)},\mathbf{d}^{(2)}):=2\sum_{v\in Q_0}d^{(1)}_vd^{(2)}_v-\sum_{a\in Q_1}d^{(1)}_{t(a)}d^{(2)}_{h(a)}-\sum_{a\in Q_1}d^{(1)}_{h(a)}d^{(2)}_{t(a)}.
$$
Write $p(\mathbf{d})=1-\frac{1}{2}(\mathbf{d},\mathbf{d})$. The fundamental region of $Q$ is the set of $0\ne \mathbf{d}\in\mathbb{Z}_{\ge0}^{Q_0}$ with connected support and with $(\mathbf{d},e_v)\le 0$ for all $v$. For any vertex $v\in Q_0$ (which should be loop-free so that $(e_v,e_v)=2$ if we work in a more general context beyond star-shaped quivers), there is a simple reflection $s_v:\mathbb{Z}^{Q_0}\rightarrow\mathbb{Z}^{Q_0}$ defined by $s_v(\mathbf{d}):=\mathbf{d}-(\mathbf{d},e_v)e_v$. The Weyl group associated to the underlying graph of $Q$ (i.e. what is obtained from $Q$ by forgetting the orientations of the arrows) is by definition the group generated by simple reflections. An element of $\mathbb{Z}^{Q_0}$ is a real root if it lies in the Weyl group orbit of a simple root, and an element of $\mathbb{Z}^{Q_0}$ is an imaginary root if it lies in the Weyl group orbit of an element of the fundamental region up to a sign. An imaginary root $\mathbf{d}$ is called isotropic if $p(\mathbf{d})=1$. Note also that $p(\mathbf{d})=0$ if $\mathbf{d}$ is a real root. The set of root $R$ consists of real roots and imaginary roots; it is the root system of the Kac-Moody Lie algebra associated to $Q$. We denote by $R^+$ the set of positive roots; that is, those roots with all coordinates nonnegative.

Define $R_{\mathbf{q}}^+:=\{\mathbf{d}\in R^+|\mathbf{q}^{\mathbf{d}}=1\}$ and
\begingroup
\allowdisplaybreaks
\begin{align*}
\Sigma_{\mathbf{q}}:=\{\mathbf{d}\in R^+_{\mathbf{q}}\mid&\text{if $\mathbf{d}=\sum_{s=1}^r\mathbf{d}^{(s)}$ with $r\ge2$ and each $\mathbf{d}^{(s)}\in R^+_{\mathbf{q}}$,}\\
&\text{then $p(\mathbf{d})>\sum_{s=1}^rp(\mathbf{d}^{(s)})$}\}.
\end{align*}
\endgroup

\begin{Con*}(\cite[Conjecture 1.4]{CB04})
Let $\mathcal{C}=(C_j)_{1\le j\le k}$ be a tuple of conjugacy classes of $\GL_n$, which defines a star-shaped quiver $Q$, a deformation parameter $\mathbf{q}$ and a dimension vector $\mathbf{d}$ as in \S \ref{subsec-Qui-Char}. Let $\Sigma_{\mathbf{q}}$ be the set of roots defined as above. Then,
the following statements are equivalent:
\begin{itemize}
\item[(i)] There is an irreducible solution to $A_1\cdots A_k=1$ with $(A_j)_j\in\mathcal{C}$. 
\item[(ii)] $\mathbf{d}\in\Sigma_{\mathbf{q}}$.
\end{itemize}
\end{Con*}

The direction (ii)$\Rightarrow$(1) has been proved by Crawley-Boevey-Shaw; see \cite[Theorem 1.1]{CBS}. The purpose of this article is to prove the other direction:
\begin{ThmA}\label{ThmA}
Suppose that there exists an irreducible solution to $A_1\cdots A_k=1$ with $(A_j)_j\in\mathcal{C}$, and that $Q$, $\mathbf{q}$ and $\mathbf{d}$ are defined by $\mathcal{C}$. Then, $\mathbf{d}$ lies in $\Sigma_{\mathbf{q}}$.
\end{ThmA}
\begin{Rem}
There is another proof of this theorem by Crawley-Boevey in his latest preprint \cite{CB25}, which  according to him has been uncirculated but complete since May 2018.
\end{Rem}

Let us remark that there are interesting variants of the Deligne-Simpson problem. We could ask whether there exist matrices $(A_j)_{1\le j\le k}$ in given adjoint orbits $\mathcal{O}_j\subset\gl_n(\mathbb{C})$ for $1\le j\le k$ satisfying:
\begin{itemize}
\item $A_1+\cdots+A_k=0$, and
\item there is no nontrivial proper vector subspace of $\mathbb{C}^n$ that is preserved by every $A_j$.
\end{itemize}
This is known as the additive Deligne-Simpson problem, and has been solved by Crawley-Boevey in \cite{CB03'}. Another variant naturally appears if we look at this problem through the Riemann-Hilbert correspondence; what we have been trying to do amounts to searching for some particular flat connections with regular singularities. However, we could also ask about connections with irregular singularities. It seems appropriate to call such a Deligne-Simpson problem wild, as opposed to the tame case considered in the present article. See Boalch \cite{Boa15}, Hiroe \cite{Hiroe}, Kulkarni-Livesay-Matherne-Nguyen-Sage \cite{KLMJNS} and Jacob-Yun \cite{JY} for various formulations and solutions in this direction.

\subsection{Strategy of the proof}\hfill

In a previous work of Schedler-Tirelli \cite{ST}, the possible dimension vectors of simple representations beyond the set $\Sigma_{\mathbf{q}}$ have been limited to types close to affine Dynkin (necessarily simply-laced). This is analogous to the results of Crawley-Boevey for additive quiver varieties; see \cite[\S 8]{CB01}, where he used some hard algebra to rule out these particular cases. In a recent preprint \cite{CB25}, he has managed to rule out these cases in the multiplicative setting, and the proof seems to be more difficult than in the additive setting. In this article, we propose to turn to geometric tools to circumvent the difficulty.

A special feature of star-shaped multiplicative quiver varieties is that they fit into the nonabelian Hodge correspondence in view of the isomorphism (\ref{eq-multi=char-2}). In some sense, this brings them closer to their additive cousins living in hyperk\"ahler geometry. It is still not known whether multiplicative quiver varieties for non-star-shaped quivers admit hyperk\"ahler structures; see however, a more general version of multiplicative quiver varieties introduced by Boalch \cite{Boa15}, as well as his conjecture \cite[\S 5]{Boa09} that these spaces are hyperk\"ahler in general.

The nonabelian Hodge correspondence changes the algebraic structure of the moduli space but preserves the stable objects, thus is fit for our purpose of finding irreducible solutions. Indeed, the use of nonabelian Hodge theory already appeared in Simpson's original paper \cite{Si91}. Another thing we can do to these moduli spaces is varying the stability conditions, or rather, the parabolic weights. It is a general fact that a generic slight perturbation of the stability condition produces a moduli space mapping to the original one, and stable objects lift to stable objects. We will use the nonabelian Hodge correspondence and variation of stability conditions to construct a sequence of maps
$$
\mathcal{M}(\overline{\mathcal{C}})\leftarrow M_1\leftarrow M_2\leftarrow\cdots\leftarrow M,
$$
until we reach a moduli space $M$ which we know well enough. We will then prove by contradiction. The existence of an irreducible local system in $\mathcal{M}(\overline{\mathcal{C}})$ will lead to a contradiction in the following two ways, depending on the combinatorics of the conjugacy classes $\mathcal{C}$:
\begin{itemize}
\item[(1)] We know that there exists no simple objects in $M$. This will be an application of a result of Kostov.
\item[(2)] The space $M$ is known to be connected. In fact, $M$ is a character variety defined by generic semi-simple conjugacy classes.
\end{itemize}
In the second case, we can moreover find a point $x\in M$ whose image $y\in\mathcal{M}(\overline{\mathcal{C}})$ is a direct sum of mutually nonisomorphic simple representations of lower dimensions. The connected component of $\mathcal{M}(\overline{\mathcal{C}})$ containing $y$ is irreducible in view of the normality proved by Kaplan-Schedler \cite{KS}. However, this component also contains the assumed simple representation as a result of (2). Thus, simple representations have to compete with direct sums of lower dimension representations for the generic point of this irreducible component, but we know from Crawley-Boevey-Shaw's theorem that the latter form a nonempty stratum and has the correct dimension, leaving no room for simple representations.

\subsection*{Acknowledgement}\hfill

I would like to thank Pengfei Huang, Emmanuel Letellier and Xueqing Wen for some helpful discussions. My thanks also go to Philip Boalch, who have oriented me through the literature on nonabelian Hodge theory. I am grateful to William Crawley-Boevey for kindly informing me of his results, encouraging me to release my preprint, and especially for developing many of the tools that are used in this paper.

\numberwithin{equation}{subsection}
\section{The proof}\label{Sec-Proof}

This section begins with a couple of ingredients that we will need in our proof of Theorem \ref{ThmA}, and the proof will be given in \S \ref{subsec-proof-ThmA}. Beware that \S \ref{subsec-reduction} only contains an overview of the crucial reduction steps, and that the details will be given in the next section.

\subsection{Normality of multiplicative quiver varieties}\hfill

Moduli spaces for 2-Calabi-Yau categories are formally locally isomorphic to formal neighbourhoods of points in additive quiver varieties. It follows from \cite{CB03} that such moduli spaces are normal. The 2-Calabi-Yau property for multiplicative preprojective algebras has been proved by Kaplan-Schedler in the cases where the quiver contains an oriented cycle. They are then able to show that for all quivers the formal neighbourhoods of multiplicative quiver varieties are isomorphic to the formal neighbourhoods of zero in some additive quiver varieties, despite the 2-Calabi-Yau property being conditional.

\begin{Thm}(\cite[Theorem 5.4]{KS})\label{Thm-Norm}
For any finite quiver, any deformation parameter $\mathbf{q}$ and any dimension vector $\mathbf{d}$, the corresponding multiplicative quiver variety $\mathcal{M}(\mathbf{q},\mathbf{d})$, if nonempty, is normal.
\end{Thm}

\subsection{Schedler-Tirelli's classification of dimension vectors of simple representations}\hfill

The first step in Crawley-Boevey's classification of dimension vectors of simple representation of deformed preprojective algebras is to show that if $\mathbf{d}$ is such a dimension vector, then one of the following occurs (see \cite[\S 7 and \S 8]{CB01}):
\begin{itemize}
\item[(i)] $\mathbf{d}\in\Sigma_{\bs\lambda}$, where $\Sigma_{\bs\lambda}$ is the additive analogue of $\Sigma_{\mathbf{q}}$.
\item[(ii)] $\mathbf{d}$ contains a multiple of the minimal positive imaginary root of an affine Dynkin diagram.
\item[(iii)] The quiver breaks into two subquivers $Q'$ and $Q''$ that are connected by a single edge, and $\mathbf{d}$ takes value one on the connecting edge.
\end{itemize}
It is easy to show that if (iii) occurs, then a $\mathbf{d}$-dimensional representation is an extension of two representations, each supported on one of the two parts $Q'$ and $Q''$, and thus is not simple. This pattern remains in the multiplicative setting. Dimension vectors in (ii) fall into the following two classes:
\begin{itemize}
\item[(1)] \textit{Isotropic imaginary roots}. Note that the support of an isotropic root is necessarily an affine Dynkin diagram, say $Q$. Let $\bs\delta$ be the minimal positive imaginary root supported on $Q$. Then, $\mathbf{d}=m\bs\delta$ for some $m\in\mathbb{Z}_{>0}$.
\item[(2)] \textit{Flat roots}. These roots are called flat in \cite{ST} because the corresponding (additive or multiplicative) moment map is flat in a neighbourhood of the given deformation parameter. Up to admissible reflections (see \S \ref{subsec-Dim-theta-st}), $\mathbf{d}$ is a root of the following form: The support of $\mathbf{d}$ is of the form $Q\sqcup J$, where $Q$ is of affine Dynkin type with affine node $0$, $J$ is a subquiver containing a vertex $\infty$, and there is a single edge connecting $Q$ and $J$ via $0$ and $\infty$. Then, $\mathbf{d}=\mathbf{d}|_J+m\bs\delta$, where $m\in\mathbb{Z}_{>0}$, $d_{\infty}=1$ and $\bs\delta$ is again the minimal positive imaginary root for the quiver $Q$. 
\end{itemize}
\begin{Rem}
If we begin with a star-shaped quiver, then we may assume $J=\{\infty\}$. Indeed, a positive root $\mathbf{d}$ either has nonincreasing value along a leg or is of the form \cite[Equation (4)]{CB04} (i.e., a root supported on a single leg); since flat roots as in (2) are obviously not of the latter form, we deduce that $d_v=1$ for any $v\in J_0$ and that $J$ is a type A quiver. We may then apply the reflections at the vertices of $J$ to reduce to the case $J=\{\infty\}$.
\end{Rem}

\begin{Thm}(\cite[Corollary 6.18]{ST})\label{Thm-dimvec-ST}
Suppose that $\mathbf{d}$ is defined on a star-shaped quiver and that there exists a simple representation in $\mathcal{M}(\mathbf{q},\mathbf{d})$. Then, one of the following occurs:
\begin{itemize}
\item $\mathbf{d}\in\Sigma_{\mathbf{q}}$.
\item[($\mathbf{Aff}$)] $\mathbf{d}\in\mathbb{Z}_{\ge2}\cdot\Sigma_{\mathbf{q}}^{iso}$, where $\Sigma_{\mathbf{q}}^{iso}\subset\Sigma_{\mathbf{q}}$ is the subset of isotropic imaginary roots.
\item[($\mathbf{Aff}^{\infty}$)] $\mathbf{d}=e_{\infty}+m\bs\delta$ is a flat root with $m\ge2$, where $e_{\infty}$ is the simple root corresponding to the vertex $\infty$. Moreover, we have $q_{\infty}=\mathbf{q}^{\bs\delta}=1$.
\end{itemize}  
\end{Thm}

It is easy to see that the dimension vectors of type ($\mathbf{Aff}^{\infty}$) do not lie in $\Sigma_{\mathbf{q}}$. The dimension vectors $\mathbf{d}$ of type ($\mathbf{Aff}$) that do not lie in $\Sigma_{\mathbf{q}}$ are of the form $m\bs\delta$, where $\mathbf{q}^{\bs\delta}=\zeta$ is an $l$-th primitive root of unity and $l<m$, while $l\bs\delta$ lies in $\Sigma_{\mathbf{q}}$ for such $\mathbf{q}$ and $l$. By \cite[Theorem 1.1]{CBS}, the stable locus of $\mathcal{M}(\mathbf{q},l\bs\delta)$ is nonempty.

Theorem \ref{ThmA} will be proved if we show that there exists no simple representation of type ($\mathbf{Aff}$) or ($\mathbf{Aff}^{\infty}$). After reducing the problem to simpler situations, as we explain in \S \ref{subsec-reduction}, we will need the results of \S \ref{subsec-Kostov} and \S \ref{subsec-conn-char} to treat ($\mathbf{Aff}$) and ($\mathbf{Aff}^{\infty}$) respectively.

\subsection{Connectedness of character varieties}\label{subsec-conn-char}\hfill

\begin{Defn}\label{Defn-gene-CC}
Let $\mathcal{C}=(C_j)_{1\le j\le k}$ be a tuple of conjugacy classes of $\GL_n$, and suppose that $C_j$ has eigenvalues $(\xi_{j,i})_{1\le i\le n}$. We say that $\mathcal{C}$ is generic if $\prod_{j=1}^k\prod_{i=1}^n\xi_{j,i}=1$ and for any $0<N<n$, any tuple of sets $(I_j)_{1\le j\le k}$ with $I_j\subset\{1,2\cdots,n\}$ and $|I_j|=N$, we have
\begin{equation}\label{eq-Defn-gene-CC}
\prod_{j=1}^k\prod_{i\in I_j}\xi_{j,i}\ne 1.
\end{equation}
\end{Defn}

The connectedness of character varieties is known under the genericity assumption on eigenvalues. Here is an orientation through the literature. If $\bar{\mathcal{C}}$ consists of generic semi-simple conjugacy classes and all eigenvalues have absolute value equal to one, then the nonabelian Hodge correspondence gives a diffeomorphism to the moduli of stable (strongly) parabolic Higgs bundles, which is well-known to be connected. For more general generic semi-simple conjugacy classes, the connectedness can be shown by counting points over finite fields; see the result of Hausel-Letellier-Rodriguez-Villegas \cite[\S 5]{HLRV}. If monodromies are allowed to have nontrivial unipotent part, then the associated character variety admits a Springer-type resolution; see, for example, \cite[\S 3.3]{Le15}. There are two ways to see that the resolution is connected. One option is again point-counting; by \cite[Theorem 3.12 and Corollary 3.14]{Le15}, the number of connected components of the resolution is equal to that of a character variety with semi-simple monodromies. Another option appeared in Ballandras' thesis work \cite{Ball}, where he showed that the resolution is diffeomorphic to a character variety with semi-simple monodromy using transcendental description of the moduli spaces. (Although we do not need Ballandras' result, our reduction precedure in \S \ref{subsec-reduction} will be a modification of his method.) 
\begin{Thm}(\cite[Theorem 5.1.1]{HLRV})\label{Thm-Conn-gene}
For any tuple of generic semi-simple conjugacy classes $\mathcal{C}$ and any genus $g\ge0$, the associated character variety $\mathcal{M}(\mathcal{C})$, if nonempty, is connected.
\end{Thm}

\subsection{The result of Kostov}\label{subsec-Kostov}\hfill

The nonexistence of simple representations in the case ($\mathbf{Aff}$) will be eventually reduced to a result of Kostov, which we translate into the quiver language below. Let $Q$ be an affine Dynkin quiver of type $\tilde{D}_4$, $\tilde{E}_6$, $\tilde{E}_7$ or $\tilde{E}_8$, and denote by $\bs\delta$ the minimal positive imaginary root in each case. Then, the following precedure recovers a tuple of conjugacy classes from an integer $m\ge1$ and a deformation parameter $\mathbf{q}\in (\mathbb{C}^{\ast})^{Q_0}$ with $q_v\ne 1$ for any $v\in Q_0$, reversing the construction in \S \ref{subset-MQV}. 

In the case of $\tilde{D}_4$, let $k=4$, and let $k=3$ in all other cases. Write $n=m\delta_{\star}$ (recall that for a star-shaped quiver, the central vertex is denoted by $\star$). We have $\delta_{\star}=2$, $3$, $4$ and $6$ in the cases $\tilde{D}_4$, $\tilde{E}_6$, $\tilde{E}_7$ or $\tilde{E}_8$ respectively, and we will construct $k$ conjugacy classes of $\GL_{n}$ in each case. For $1\le j\le k$, choose complex numbers $\xi_{[j,0]}$ such that $\prod_{j=1}^k\xi_{[j,0]}=q_{\star}$ and define $\xi_{[j,i]}=q_{[j,i]}\xi_{[j,i-1]}$ for $1\le i\le \nu_j$; let $C_j$ be the semi-simple conjugacy class with eigenvalues $\xi_{[j,i]}$, $0\le i\le \nu_j$, each having multiplicity $m(\delta_{[j,i]}-\delta_{[j,i+1]})$. In the case of $\tilde{D}_4$ and $\tilde{E}_6$, every eigenvalue has multiplicity $m$. In the case of $\tilde{E}_7$, two conjugacy classes only have multiplicity-$m$ eigenvalues, while the remaining one has two eigenvalues of multiplicity $2m$. In the case of $\tilde{E}_8$, one conjugacy class only has multiplicity-$m$ eigenvalues, one conjugacy class only has multiplicity-$2m$ eigenvalues, and the remaining one has two eigenvalues of multiplicity $3m$. The resulting conjugacy classes for difference choices of the $\xi_{[j,0]}$ only differ by scalar matrices, and the solvability of the Deligne-Simpson problem is not affected. The above explicit description of these conjugacy classes shows that they satisfy:
\begin{equation}\label{eq-kappa=0}
\sum_{j=1}^kC_j=2n^2,
\end{equation}
which is equivalent to the condition $\kappa=0$ in \cite{Kos01}. Since $\mathbf{q}^{m\bs\delta}=1$, the number $\zeta:=\mathbf{q}^{\bs\delta}$ is an $m$-th root of unity. Denote the order of $\zeta$ by $l$.
\begin{Defn}
We say that the conjugacy classes $(C_j)_{1\le j\le k}$ are \textit{almost generic} if the only dimension vectors $0<\bs\gamma\le m\bs\delta$ satisfying $\mathbf{q}^{\bs\gamma}=1$ are of the form $\bs\gamma=m'l\bs\delta$ for some $m'\in\mathbb{Z}_{>0}$.
\end{Defn}
Almost generic conjugacy classes are called \textit{relatively generic} by Kostov, and he calls the equality $\mathbf{q}^{l\bs\delta}=1$ \textit{the basic nongenericity relation}; see \cite[\S 2.2, Remark 8 and Definition 9]{Kos01}. To compare this definition with Definition \ref{Defn-gene-CC}, we remark that $\gamma_{\star}$ should be thought of as the number $N$ there, and (\ref{eq-Defn-gene-CC}) amounts to saying the no relation of the form $\mathbf{q}^{\bs\gamma}=1$ is allowed.
\begin{Thm}(\cite[Theorem 15]{Kos01})\label{Thm-Kos}
Suppose that $m\ge 2$, $l<m$, and that the semi-simple conjugacy classes $(C_j)_{1\le j\le k}$ defined above satisfy (\ref{eq-kappa=0}) and are almost generic. Then, there exists no irreducible solution to $\prod_{j=1}^kA_j=\Id$ with $A_j\in C_j$.
\end{Thm}
For not necessarily semi-simple conjugacy classes, Kostov also has a conditional nonexistence result; see \cite[Theorem 29]{Kos01}.

\subsection{Reduction to semi-simple character varieties}\label{subsec-reduction}\hfill

Let us recall the setting that we will be working in. In case ($\mathbf{Aff}$), the quiver $Q$ is an affine Dynkin graph, $\bs\delta$ denotes the minimal positive imaginary root, $\mathbf{d}=m\bs\delta$ and $\mathbf{q}^{\bs\delta}$ is a primitive $l$-th root of unity with $l<m$. In case ($\mathbf{Aff}^{\infty}$), there is an extra vertex $\infty$ connected to the affine node $0$ by a single edge; moreover, $d_{\infty}=1$, $\mathbf{d}|_{Q}=m\bs\delta$ with $m\ge2$ and $q_{\infty}=\mathbf{q}^{\bs\delta}=1$. The simple root corresponding to $\infty$ is $e_{\infty}$, and we denote by $\rho_{\infty}$ the simple representation of dimension $e_{\infty}$. The main difference between ($\mathbf{Aff}$) and ($\mathbf{Aff}^{\infty}$) is that in the former case $\mathbf{d}$ is divisible (i.e., there exists $\mathbf{d}'\in\mathbb{Z}_{\ge0}^{Q_0}$ and $n'\in\mathbb{Z}_{\ge2}$ such that $\mathbf{d}=n'\mathbf{d}'$), whereas in the latter case $\mathbf{d}$ is not. Divisibility determines whether there exist generic stability conditions.

The reduction of the problem will be achieved by  constructing the following diagram
\begin{equation}\label{eq-red-diag}
\begin{tikzcd}[row sep=2.5em, column sep=2.5em]
\mathcal{M}(\mathbf{q},\mathbf{d}) \arrow[d, equal, "\mathbf{Iso}_1"'] & \mathcal{M}_{\bs\theta}(\mathbf{q},\mathbf{d}) \arrow[l, "\mathbf{Var}_0"'] \arrow[d, equal, "\mathbf{Iso}_2"] & \\
\mathcal{M}_B(\overline{\mathcal{C}}) & \mathcal{M}_B(\mathbf{d},\bs\beta,\bs\xi) \arrow[l, "\mathbf{Var}_1"] \arrow[r, "\mathbf{NH}_1"] & \mathcal{M}_{Dol}(\mathbf{d}',\bs\alpha',(\mathcal{O}'_j)_j) \\
\mathcal{M}_B(\mathcal{C}') & \mathcal{M}_B(\mathbf{d},\bs\beta,\bs\xi') \arrow[l, "\mathbf{Iso}_3"] \arrow[r, "\mathbf{NH}_2"'] & \mathcal{M}_{Dol}(\tilde{\mathbf{d}},\bs\alpha,(\mathcal{O}_j)_j). \arrow[u, "\mathbf{Var}_2"']
\end{tikzcd}
\end{equation}
The construction of the these spaces and morphisms will occupy the entire \S \ref{Sec-Reduction}; however, let us explain the first step $\mathbf{Var}_0$. 

\begin{itemize}
\item $\mathcal{M}(\mathbf{q},\mathbf{d})$ is a multiplicative quiver variety of type ($\mathbf{Aff}$) or ($\mathbf{Aff}^{\infty}$) as in Theorem \ref{Thm-dimvec-ST}.
\item $\mathcal{M}_{\bs\theta}(\mathbf{q},\mathbf{d})$ is a multiplicative quiver variety with a stability condition $\bs\theta$.
\item $\mathbf{Var}_0$ is the morphism induced by varying the stability condition.
\end{itemize}
The requirement on $\bs\theta$ is as follows. In case ($\mathbf{Aff}^{\infty}$), by \cite[Theorem 6.23]{ST}, for a generic $\bs\theta$ the morphism $\mathbf{Var}_0$ is a resolution. For our purpose, we need to choose $\bs\theta$ in such a way that $\theta_v>0$ for every $v\ne\star$; this is to guarantee that the resulting multiplicative quiver variety is isomorphic to a suitable moduli space of filtered local systems. But this is no serious restriction, and we can choose such a $\bs\theta$ which is also generic. In case ($\mathbf{Aff}$), we simply choose $\bs\theta$ which is strictly positive away from $\star$ with $\bs\theta\cdot\bs\delta=0$.

The rest of diagram (\ref{eq-red-diag}) consists of the following objects. 
\subsubsection*{Betti moduli spaces}
Filtered local systems and their moduli spaces will be defined in \S \ref{subsec-FLS} below.
\begin{itemize}
\item $\mathcal{M}_B(\overline{\mathcal{C}})$ is the character variety with monodromies in closures of conjugacy classes.
\item $\mathcal{M}_B(\mathbf{d},\bs\beta,\bs\xi)$ is a moduli space of filtered local systems. In the associated graded of the filtered structures, the monodromies are scalar matrices; however, the same eigenvalue may appear in different graded spaces.
\item $\mathcal{M}_B(\mathbf{d},\bs\beta,\bs\xi')$ is another moduli space of filtered local systems. The monodromies are again scalar matrices in the associated graded, but the eigenvalues can distinguish different graded spaces.
\item $\mathcal{M}_B(\mathcal{C}')$ is a character varieties with monodromies in semi-simple conjugacy classes.
\end{itemize}

\subsubsection*{Dolbeault moduli spaces}

Parabolic Higgs bundles and their moduli spaces will be defined in \S \ref{subsec-PHB} below.
\begin{itemize}
\item $\mathcal{M}_{Dol}(\mathbf{d}',\bs\alpha',(\mathcal{O}'_j)_j)$ is a moduli space of parabolic Higgs bundles. The residues of the Higgs fields lie in prescribed semi-simple adjoint orbits when passing to the Levi quotients of the parabolic structures, but these adjoint orbits need not be central in the Levi.
\item $\mathcal{M}_{Dol}(\tilde{\mathbf{d}},\bs\alpha,(\mathcal{O}_j)_j)$ is another moduli space of parabolic Higgs bundles. The parabolic structures refine those in $\mathcal{M}_{Dol}(\mathbf{d}',\bs\alpha',(\mathcal{O}'_j)_j)$ in such a way that the residues of the Higgs fields are central when passing to the Levi quotients.
\end{itemize}

\subsubsection*{Morphisms}
\begin{itemize}
\item $\mathbf{Iso}_1$ and $\mathbf{Iso}_2$ are the identification between multiplicative quiver varieties and the moduli spaces of (filtered) local systems. More precisely, $\mathbf{Iso}_1$ is the isomorphism (\ref{eq-multi=char-2}), and $\mathbf{Iso}_2$ is given by \cite[Theorem 1.2]{Yama} where we need $\theta_v>0$ for all $v\ne\star$. The dimension vectors on the two sides of $\mathbf{Iso}_2$ are the same in a suitable sense; see \cite[Theorem 4.16]{Yama} and \S \ref{subsec-Comb-data}. 
\item $\mathbf{Var}_1$ is induced by the functor of forgetting the filtered structures; it is compatible with $\mathbf{Var}_0$.
\item $\mathbf{NH}_1$ and $\mathbf{NH}_2$ are nonabelian Hodge correspondences, which are homeomorphisms of topological spaces that preserve stable objects; see \S \ref{subsec-NH} below.
\item $\mathbf{Var}_2$ is induced by variation of weights of parabolic Higgs bundles; see \S \ref{subsec-VPW} below. In particular, the weight $\bs\alpha$ is more generic than $\bs\alpha'$.
\item $\mathbf{Iso}_3$ is an isomorphism of Betti moduli spaces; see \S \ref{subsec-FFS} below. The key assumption behind the statement is that the eigenvalues are (almost) generic, which is a consequence of our choice of $\bs\alpha$.
\end{itemize}

\subsection{The proof of Theorem \ref{ThmA}}\label{subsec-proof-ThmA}\hfill

We will need to identify multiplicative quiver varieties and character varieties in our proof; however, not every multiplicative quiver variety for a star-shaped quiver is isomorphic to a character variety. As we have seen, the relevant quivers are star-shaped affine Dynkin diagram, possibly with an extra vertex joined to the affine node. It is easy to verify that in every case that concerns us, the dimension vector satisfies \cite[Equation (7)]{CB04}; that is, the integers $d_{[j,i-1]}-d_{[j,i]}$ for vertices $[j,i]$ with equal defomation parameters are nonincreasing along the legs of the star-shaped quiver. It follows that such multiplicative quiver varieties are indeed isomorphic to character varieties.

Our goal is show that either in case ($\mathbf{Aff}$) or ($\mathbf{Aff}^{\infty}$), there is no simple representation of dimension $\mathbf{d}$. We will prove by contradiction; therefore, we consider a hypothetical simple representation $\rho_s\in \mathcal{M}(\mathbf{q},\mathbf{d})$ either in case ($\mathbf{Aff}^{\infty}$) or ($\mathbf{Aff}$). A simple representation is necessarily $\bs\theta$-stable. In either case, there is a point $\rho'_s\in\mathcal{M}_{\bs\theta}(\mathbf{q},\mathbf{d})$ such that $\mathbf{Var}_0(\rho'_s)=\rho_s$. 

\textit{The case $(\mathbf{Aff})$}. Regard $\rho'_s$ as an element of $\mathcal{M}_B(\mathbf{d},\bs\beta,\bs\xi)$. Then, $\mathbf{NH}_1$ sends it to a $\bs\alpha'$-stable parabolic Higgs bundle $(E'_s,\Phi'_s)$. We will show in Proposition \ref{Prop-Var2} that there exists an $\bs\alpha$-stable parabolic Higgs bundle $(E_s,\Phi_s)$ that is mapped to $(E'_s,\Phi'_s)$ by $\mathbf{Var}_2$; here, $\bs\alpha$ is chosen to be almost generic (see Definition \ref{defn-almost-gen-w}). Passing to the Betti side of $\mathbf{NH}_2$, we find a $\bs\beta$-stable filtered local system $\mathscr{L}$, where the eigenvalues $\bs\xi'$ are almost generic. Then, Proposition \ref{Prop-Iso3} will show that $\mathbf{Iso}_3(\mathscr{L})$ is an irreducible local system with monodromies in $\mathcal{C}'$. However, this contradicts Kostov's Theorem \ref{Thm-Kos}. Indeed, the conjugacy classes $\mathcal{C}'$ satisfy the assumptions in Kostov's theorem by Lemma \ref{Lem-Betti-Dol-indiv}, Remark \ref{Rem-Betti-Dol-indiv} and Lemma \ref{Lem-almost-Dol-to-Betti}. We conclude that there exists no simple representation in $\mathcal{M}(\mathbf{q},\mathbf{d})$.

\textit{The case $(\mathbf{Aff}^{\infty})$}. The theorem of Kostov does not apply in this case, since the semi-simple conjugacy classes $\mathcal{C}'$ are generic. We will therefore adopt a different strategy. As we have seen, both $\mathbf{Var}_0$ and $\mathbf{Var}_1$ are surjective in this case. Proposition \ref{Prop-Var2-inf} will show that $\mathbf{Var}_2$ is surjective. Since $\bs\alpha$ is chosen to be generic, the corresponding eigenvalues $\bs\xi'$ are also generic. By Proposition \ref{Prop-Iso3}, the morphism $\mathbf{Iso}_3$ is an isomorphism. By Theorem \ref{Thm-Conn-gene}, the character variety $\mathcal{M}_B(\mathcal{C}')$ is connected. It follows that every moduli space in the diagram (\ref{eq-red-diag}) is connected. By Theorem \ref{Thm-Norm}, $\mathcal{M}(\mathbf{q},\mathbf{d})$ is irreducible.

Let $\mathbf{d}=\sum_{s=1}^r\mathbf{d}^{(s)}$ be any decomposition of $\mathbf{d}$ into vectors $\mathbf{d}^{(s)}\in R^+_{\mathbf{q}}$. We have 
\begin{equation}\label{eq-Lem-dim-Aff-infty}
p(\mathbf{d})\ge \sum_{s=1}^rp(\mathbf{d}^{(s)}).
\end{equation}
This is in fact the defining property of flat roots in \cite{ST}, and it is a part of the statement of \cite[Theorem 6.16]{ST}. (The proof of this theorem is based on \cite{Su}, and what we actually need here is \cite[Proposition 4.2]{Su}.)
By \cite[Lemma 7.1]{CBS}, the locus of semi-simple representations of type $(k_i,\mathbf{d}^{(i)})_{1\le i\le r}$, if nonempty, has dimension $\sum_{i=1}^r2p(\mathbf{d}^{(i)})$. It follows from the inequality (\ref{eq-Lem-dim-Aff-infty}) that every representation type stratum has dimension at most $2p(\mathbf{d})$. The locus consisting of direct sums of distinct $\bs\delta$-dimensional simple representations has dimension 2m and is nonempty in view of Crawley-Boevey-Shaw's theorem \cite[Theorem 1.1]{CBS}. Now, the irreducibility of $\mathcal{M}(\mathbf{q},\mathbf{d})$ implies that there is only one stratum of dimension $2m$. Therefore, we have $\dim\mathcal{M}(\mathbf{q},\mathbf{d})=2p(\mathbf{d})=2m$. Since the locus of simple representations also has dimension $2p(\mathbf{d})=2m$ if it is nonempty, we conclude that $\mathbf{d}$-dimensional simple representations do not exist.

\section{The reduction steps}\label{Sec-Reduction}

This section contains the details of the reduction steps in \S \ref{subsec-reduction}.

\subsection{Combinatorial data}\label{subsec-Comb-data}\hfill

We begin by introducing some combinatorial data that will be used to describe conjugacy classes, parabolic structures and quivers in the rest of this article.

Let $k$ be a positive integer, and let $\underline{\nu}:=(\nu_j)_{1\le k\le j}\in\mathbb{Z}_{\ge0}^k$ be a tuple of integers. Define
$$
\bs\tau(\underline{\nu}):=\{[j,i]|1\le j\le k,~0\le i\le \nu_j\}
$$
where $[j,i]$ means a pair consisting of integers $j$ and $i$ (instead of an interval); we will call $\bs\tau(\underline{\nu})$ a type, which will be associated to a filtered structure or a parabolic structure. An element of $\mathbb{Z}_{\ge0}^{\bs\tau(\underline{\nu})}$ will be written as $\mathbf{d}=(d_{[j,i]})_{[j,i]\in\bs\tau(\underline{\nu})}$ with $d_{[j,i]}\in\mathbb{Z}_{\ge0}$, and similarly $\mathbf{q}=(q_{[j,i]})_{[j,i]\in\bs\tau(\underline{\nu})}\in(\mathbb{C}^{\ast})^{\bs\tau(\underline{\nu})}$. For $\mathbf{d}\in\mathbb{Z}_{\ge0}^{\bs\tau(\underline{\nu})}$, it will be convenient to introduce the associated vector
\begin{equation}\label{eq-asso-d}
\mathbf{d}^{\ast}=(d_{[j,i]}^{\ast})_{[j,i]\in\bs\tau(\underline{\nu})}
\end{equation}
with $d_{[j,i]}^{\ast}=d_{[j,i]}-d_{[d,i+1]}$ for all $[j,i]$. Following Crawley-Boevey, we say that $\mathbf{d}$ is strict if every component of $\mathbf{d}^{\ast}$ is nonnegative. We define a star-shaped quiver $Q=(Q_0,Q_1)$ from $\bs\tau(\underline{\nu})$ in the following way. Let $Q'_0=\bs\tau(\underline{\nu})$ and let $Q'_1$ be the set of arrows $a_{[j,i]}:[j,i]\rightarrow[j,i+1]$ for $[j,i]\in\bs\tau(\underline{\nu})$ and $i\ne\nu_j$. We define $Q_0$ as $Q'_0$ modulo the relation $[j,0]\sim[j',0]$ for any $1\le j\le j'\le k$ and define $Q_1$ as the induced set of arrows among the elements of $Q_0$. We will say that $\mathbf{d}$ has rank $n$ if $d_{[j,0]}=n$ for all $j$ and it is strict; we may regard such a $\mathbf{d}$ as an element of $\mathbb{Z}_{\ge0}^{Q_0}$. Conversely, any strict $\mathbf{d}\in\mathbb{Z}_{\ge0}^{Q_0}$ can be regarded as a rank $n$ vector in $\mathbb{Z}_{\ge0}^{\bs\tau(\underline{\nu})}$.

We will need to consider an operation on flags, called degeneration, which send
$$
\mathbb{C}^n=E_0\supset E_1\supset\cdots\supset E_{\nu}\supset E_{\nu+1}=0
$$ 
to
$$
E_0=E_{i_0}\supset E_{i_1}\supset\cdots\supset E_{i_s}\supset E_{i_s+1}=0
$$
where $\{i_1,\ldots,i_s\}$ is a subset of $\{1,2,\ldots,\nu\}$. The following notations are introduced for this purpose. Consider a tuple of integers $\underline{\mu}=(\mu_j)_{1\le j \le k}\in\mathbb{Z}_{\ge0}^k$ with $\mu_j\le \nu_j$ for every $j$. A degeneration of types is a map
$$
\sigma:\bs\tau(\underline{\mu})\longrightarrow\bs\tau(\underline{\nu})
$$ 
that preserves the $j$-component while being increasing in the $i$-component and satisfies $\sigma([j,0])=[j,0]$ for every $j$. By abuse of notation, we may write $\sigma([j,i])=[j,\sigma(i)]$. For any $\mathbf{d}\in\mathbb{Z}_{\ge0}^{\bs\tau(\underline{\nu})}$, the value of $\sigma^{\ast}\mathbf{d}$ at $[j,i]$ is equal to $d_{[j,\sigma(i)]}$; this should be thought of as the dimension vector of a degenerated flag.

\subsection{Filtered local systems}\label{subsec-FLS}\hfill

Let $\bar{C}$ be a compact Riemann surface of genus $g\ge0$ and let $S=\{p_1,\ldots,p_k\}$ be a set of points in $\bar{C}$. Write $C=\bar{C}\setminus S$. For any $j\in\{1,\ldots,k\}$, choose a simply connected analytic open neighbourhood $U_j$ of $p_j$, and write $U_j^{\ast}=U_j\setminus\{p_j\}$; the choice of such opens will not matter. Let $L$ be a local system on $C$. A filtered structure on $L$ is a tuple $(L_{[j,\bullet]})_{1\le j\le k}$, where each $L_{[j,\bullet]}$ is a strictly decreasing filtration of $L|_{U_j^{\ast}}$:
$$
L|_{U_j^{\ast}}=L_{[j,0]}\supset L_{[j,1]}\supset \cdots\supset L_{[j,\nu_j]}\supset L_{[j,\nu_j+1]}=0.
$$
A local system on $C$ equipped with a filtered structure is called a filtered local system, denoted by $\mathscr{L}=(L,L_{[\bullet,\bullet]})$. We will call $\bs\tau(\underline{\nu})$ its type. The dimension vector of this filtered structure is a tuple of integers $\mathbf{d}\in\mathbb{Z}_{\ge0}^{\bs\tau(\underline{\nu})}$ with $d_{[j,i]}=\rk L_{[j,i]}$. A weight for such a filtered structure (or such a filtered local system) is a tuple of real numbers $\bs\beta\in\mathbb{R}^{\bs\tau(\underline{\nu})}$ satisfying
$$
\beta_{[j,0]}<\beta_{[j,1]}<\cdots<\beta_{[j,\nu_j]}
$$
for all $j$. The weight allows us to define the degree of filtered local systems:
$$
\deg_{\bs\beta}\mathscr{L}:=\sum_{j=1}^k\sum_{i=0}^{\nu_j}\beta_{[j,i]}\dim L_{[j,i]}/L_{[j,i+1]}.
$$ 
Any local subsystem $M\subset L$ admits an induced filtered structure in the following manner.  For any $j$, the induced filtration $M_{[j,\bullet]}$ consists of the distinct vector spaces among $M_{p_j}\cap L_{[j,\bullet]}$; if $\mu_j$ denotes the number of such distinct spaces that are nonzero, then the induced filtered local subsystem has type $\bs\tau(\underline{\mu})$. For each $[j,i_1]\in\bs\tau(\underline{\mu})$, if $[j,i_2]\in\bs\tau(\underline{\nu})$ is such that $L_{[j,i_2]}\supset M_{[j,i_1]}$ and $L_{[j,i_2+1]}\nsupset M_{[j,i_1]}$, then we assign the weight $\beta_{[j,i_2]}$ to $M_{[j,i_1]}$. This defines the weight for $\mathscr{M}:=(M,M_{[\bullet,\bullet]})$, an element of $\mathbb{R}^{\bs\tau(\underline{\mu})}$, which by abuse of notation we again denote by $\bs\beta$. A filtered local system $\mathscr{L}$ of degree zero is $\bs\beta$-stable (resp. $\bs\beta$-semi-stable) if for any local subsystem $M\subset L$, we have
$$
\deg_{\bs\beta}\mathscr{M}<0 \text{ (resp. $\le0$).}
$$
According to Yamakawa \cite[Theorem 1.2]{Yama} and Huang-Sun \cite[Theorem 1.1]{HS}, for 
\begin{itemize}
\item any positive intergers $n$ and $k$,
\item any subset $S\subset\bar{C}$ consisting of $k$ points, 
\item any type of filtration $\bs\tau(\underline{\nu})$, dimension $\mathbf{d}$ of rank $n$ with $d_{[j,i]}>d_{[j,i+1]}$ for any $[j,i]$,
\item any weight $\bs\beta$ of type $\bs\tau(\underline{\nu})$, and
\item any tuple of complex number $\bs\xi\in(\mathbb{C}^{\ast})^{\bs\tau(\underline{\nu})}$ satisfying 
$$
\bs\xi^{\mathbf{d}}:=\prod_{1\le i\le k}(\xi_{[j,0]})^n\prod_{\substack{1\le j\le k\\1\le i\le \nu_j}}(\xi_{[j,i]}\xi_{[j,i-1]}^{-1})^{d_{[j,i]}}=1,
$$
\end{itemize}
there exists a coarse moduli space $$\mathcal{M}_B(S,\mathbf{d},\bs\beta,\bs\xi)$$ parametrising $\bs\beta$-polystable filtered local system of degree zero on $C$ with the given filtration type $\bs\tau(\underline{\nu})$ and dimension $\mathbf{d}$, whose monodromy at $p_j$ induces the scalar matrix $\xi_{[j,i]}$ on each graded local system $L_{[j,i]}/L_{[j,j+1]}$. We will often write $\mathcal{M}_B(\mathbf{d},\bs\beta,\bs\xi)=\mathcal{M}_B(S,\mathbf{d},\bs\beta,\bs\xi)$.

\subsection{Parabolic Higgs bundles}\label{subsec-PHB}\hfill

Now we regard $\bar{C}$ as a smooth projective algebraic curve. Let $S=\{p_1,\ldots,p_k\}\subset \bar{C}$ be as above, regarded as a reduced divisor. Denote by $\Omega$ the canonial bundle of $\bar{C}$. A meromorphic Higgs bundle of rank $n$ on $\bar{C}$ is a pair $(E,\Phi)$, where $E$ is a vector bundle of rank $n$ on $\bar{C}$ and $\Phi:E\rightarrow E\otimes \Omega(S)$ is a homomorphism of coherent sheaves. For each $1\le j\le k$, we will denote by $\Phi_j$ the residue of the Higgs field $\Phi$ at $p_j$. A parabolic structure on $(E,\Phi)$ is the data of a flag $E_{[j,\bullet]}$ of the fibre $E_{p_j}$ that is preserved by $\Phi_j$ for each $1\le j\le k$. A meromorphic Higgs bundle equipped with a parabolic structure is called a parabolic Higgs bundle, denoted by $\mathscr{E}=(E,E_{[\bullet,\bullet]},\Phi)$. If the parabolic sturcture on $(E,\Phi)$ is of the form
$$
E_{p_j}=E_{[j,0]}\supset E_{[j,1]}\supset\cdots\supset E_{[j,\nu_j]}\supset E_{[j,\nu_j+1]}=0
$$ 
for each $j$, then we say that it is of type $\bs\tau(\underline{\nu})$. The dimension vector of this parabolic structure is a tuple of integers $\mathbf{d}\in\mathbb{Z}_{\ge0}^{\bs\tau(\underline{\nu})}$ with $d_{[j,i]}=\rk E_{[j,i]}$. A parabolic weight (or simply weight) for such a parabolic Higgs bundle is a collection of real number $\bs\alpha\in\mathbb{R}^{\bs\tau(\underline{\nu})}$ satisfying
$$
0\le \alpha_{[j,0]}<\alpha_{[j,1]}<\cdots<\alpha_{[j,\nu_j]}<1
$$
for each $j$. Any vector subbundle $F\subset E$ that is preserved by $\Phi$ admits an induced parabolic structure in the following manner.  For any $j$, the induced filtration $F_{[j,\bullet]}$ consists of the distinct vector spaces among $F_{p_j}\cap E_{[j,\bullet]}$; if $\mu_j$ denotes the number of such distinct spaces that are nonzero, then the induced parabolic Higgs bundle has type $\bs\tau(\underline{\mu})$. For each $[j,i_1]\in\bs\tau(\underline{\mu})$, if $[j,i_2]\in\bs\tau(\underline{\nu})$ is such that $E_{[j,i_2]}\supset F_{[j,i_1]}$ and $E_{[j,i_2+1]}\nsupset F_{[j,i_1]}$, then we assign the weight $\alpha_{[j,i_2]}$ to $F_{[j,i_1]}$. This defines the weight for $\mathscr{F}:=(F,F_{[\bullet,\bullet]},\Phi)$, an element of $\mathbb{R}^{\bs\tau(\underline{\mu})}$, which by abuse of notation we again denote by $\bs\alpha$.

The parabolic degree of $\mathscr{E}$ with weight $\bs\alpha$ is defined to be
$$
\deg_{\bs\alpha}\mathscr{E}:=\deg E+\sum_{j=1}^k\sum_{i=0}^{\nu_j}\alpha_{[j,i]}\dim E_{[j,i]}/E_{[j,i+1]}.
$$
We say that a parabolic Higgs bundle $\mathscr{E}$ of degree zero is $\bs\alpha$-stable (resp. $\bs\alpha$-semi-stable) if for any vector subbundle $F\subset E$, we have
$$
\deg_{\bs\alpha}\mathscr{F}<0\text{ (resp. $\le0$)}.
$$
According to \cite[Corollary 7.2]{HKSZ}, for any
\begin{itemize}
\item any $n$, $k$, $S$, $\mathbf{d}$ and $\bs\tau(\underline{\nu})$ as in \S \ref{subsec-FLS},
\item any $\bs\alpha\in(\mathbb{R}_{\ge0}\cap\mathbb{R}_{<1})^{\bs\tau(\underline{\nu})}$, and
\item any tuple $(\mathcal{O}_j)_{1\le j\le k}$ of semi-simple adjoint orbits in the Lie algebras $\bigoplus_{i=0}^{\nu_j}\gl(E_{[j,i]}/E_{[j,i+1]})$, each regarded as the Levi quotient of the parabolic subalgebra defined by the parabolic structure on $E_{p_j}$,
\end{itemize}
there exists a coarse moduli space
$$
\mathcal{M}_{Dol}(S,\mathbf{d},\bs\alpha,(\mathcal{O}_j)_j)
$$
parametrising $\bs\alpha$-polystable parabolic Higgs bundles of degree zero on $\bar{C}$ with the given filtration type $\bs\tau(\underline{\nu})$ and dimension $\mathbf{d}$, whose Higgs fields have residues at $p_j$'s lying in the given orbits $\mathcal{O}_j$ after passing to the Levi quotients. We will often write $\mathcal{M}_{Dol}(\mathbf{d},\bs\alpha,(\mathcal{O}_j)_j)=\mathcal{M}_{Dol}(S,\mathbf{d},\bs\alpha,(\mathcal{O}_j)_j)$.

\subsection{Nonabelian Hodge theory}\label{subsec-NH}\hfill

The nonabelian Hodge correspondence between the moduli space of filtered local systems and the moduli space of parabolic Higgs bundles was established by Simpson in \cite{Si90}, where he proved a bijection between the stable objects on each side. It is possible to extend this bijection to a homeomorphism between the full moduli spaces. For the following theorem, see Biquard-Garc\'ia-Prada-Mundet i Riera \cite[Theorem 7.10]{BGM} (the published version) or Huang-Sun's recent preprint \cite[Theorem 5.3]{HS}. Note that what is usually called the nonabelian Hodge correspondence is a map between the Dolbeault moduli space and the de Rham moduli space, and in the theorem below we have tacitly added to it the Riemann-Hilbert correspondence. As is pointed out to me by Boalch, the Riemann-Hilbert is nontrivial for all the noncompact generalisations. To simplify notations and terminologies, we will nevertheless hide the Riemann-Hilbert in the nonabelian Hodge. For the actual nonabelian Hodge correspondence, it is known since the earlier work of Biquard-Boalch \cite[Theorem 6.1]{BB} that it induces a diffeomorphism between the stable loci of moduli spaces. Some special cases of this diffeomorphism were also proved by Konno \cite{Kon} and Nakajima \cite{Naka}.  
\begin{Thm}\label{Thm-NHT}
The nonabelian Hodge correspondence induces a homeomorphism of topological spaces
$$
\mathcal{M}_B(\mathbf{d}^{(1)},\bs\beta,\bs\xi)\lisom\mathcal{M}_{Dol}(\mathbf{d}^{(2)},\bs\alpha,(\mathcal{O}_j)_j),
$$
for suitable $\mathbf{d}^{(1)}$, $\bs\beta$, $\bs\xi$, $\mathbf{d}^{(2)}$, $\bs\alpha$ and $(\mathcal{O}_j)_j$.
\end{Thm}
For our purpose, it is crucial to make precise the relation between the invariants on each side of of the homeomorphism. Recall that $\mathbf{d}^{(1)}$ and $\mathbf{d}^{(2)}$ are the dimension vectors of the filtered and parabolic structures respectively; their relation can be read from the transformation rule for other invariants:

\begin{equation}
\begin{tabular}{|c|c|c|}
\hline
 & Dolbeault & Betti \\
\hline
weight & $\alpha$ & $\beta=-2b$ \\
\hline
eigenvalue & $b+ci$ & $\exp(-2\pi i\alpha+4\pi c)$\\
\hline
\end{tabular}
\end{equation}

In this table, we let $\alpha$ and $\beta$ denote a respective component of $\bs\alpha$ and $\bs\beta$. Let us explain how to read this table. Suppose that we have a filtered local system $(L,L_{[j,\bullet]})$, and that at some puncture $p_j$, the graded monodromy of $L$ is given by a tuple of scalar matrices:
$$
(\xi_{[j,i]}\Id)_{0\le i\le \nu_j}\in\bigoplus_{i=0}^{\nu_j^{(1)}} \gl(L_{[j,i]}/L_{[j,i+1]}).
$$
In other words, for any $[j,i]$, there is a scalar matrix $\xi\Id:=\xi_{[j,i]}\Id$ of size $d_{[j,i]}-d_{[j,i+1]}$ and weight $\beta=\beta_{[j,i]}$. We may write $\beta=-2b$ and $\xi=\exp(-2\pi i\alpha+4\pi c)$ for some real numbers $0\le\alpha<1$, $b$ and $c$; these numbers are the inputs on the Betti side of the above table. Then, the rule of transformation says that the corresponding parabolic Higgs bundle will have the following form: if we write the graded of the residue of the Higgs field at $p_j$ as
$$
(\Phi_{[j,i]})_{0\le i\le \nu_j^{(2)}}\in\bigoplus_{i=0}^{
\nu_j^{(2)}}\gl(E_{[j,i]}/E_{[j,j+1]}),
$$
then, 
\begin{itemize}
\item[(1)] the set of weights $\{\alpha_{[j,i]}\}_i$ for the flag $E_{[j,\bullet]}$ consists of those $\alpha$ such that $-2\pi\alpha$ is the argument of some eigenvalue $\xi$,
\item[(2)] for a subquotient $E_{[j,i]}/E_{[j,j+1]}$ of weight $\alpha$, the dimension $\dim E_{[j,i]}/E_{[j,j+1]}$ is computed by collecting all scalar matrices $\xi\Id$ on the Betti side with the argument of $\xi$ equal to $-2\pi\alpha$ and then summing up the sizes of these matrices, 
\item[(3)] each component $\Phi_{[j,i]}$ is semi-simple, and
\item[(4)] whenever a matrix $\xi\Id$ of weight $\beta$ appears on the Betti side, the component $\Phi_{[j,i]}$ of weight $\alpha$ has an eigenvalue $(b+ci)$ with multiplicity being the size of the matrix $\xi\Id$. 
\end{itemize}

In the diagram (\ref{eq-red-diag}), there were two arrows defined by nonabelian Hodge theory:
\begin{itemize}
\item[(1)] $\mathbf{NH}_1$ is precisely the correspondence described above. Note that in general $\nu_j^{(2)}\le\nu_j^{(1)}$ and the strict inequality occurs for some $j$ precisely when some component $\Phi_{[j,i]}$ is not a scalar matrix.
\item[(2)] $\mathbf{NH}_2$ is similar, but the Dolbeault side of this correspondence, which we will explain in the next subsection, is defined in such a way that each component $\Phi_{[j,i]}$ is a scalar matrix.
\end{itemize}

\begin{Lem}\label{Lem-Betti-Dol-indiv}
Consider the morphisms
$$
\mathcal{M}_{\bs\theta}(\mathbf{q},\mathbf{d})\stackrel{\mathbf{Iso}_2}{\longrightarrow}\mathcal{M}_B(\mathbf{d},\bs\beta,\bs\xi)\stackrel{\mathbf{NH}_1}{\longrightarrow}\mathcal{M}_{Dol}(\mathbf{d}',\bs\alpha',(\mathcal{O}'_j)_j),
$$
where the objects on the Betti and Dolbeault side of $\mathbf{NH}_1$ have types $\bs\tau(\underline{\nu})$ and $\bs\tau(\underline{\mu})$ respectively. Suppose that $\mathbf{d}=m\bs\gamma$ for some indivisible $\bs\gamma$ and $\mathbf{q}^{\bs\gamma}$ is an $l$-th primitive root of unity. Let $e\in\mathbb{Z}$ be such that 
$$
e+\sum_{[j,i]\in\bs\tau(\underline{\mu})} \alpha'_{[j,i]}d^{\prime\ast}_{[j,i]}=0.
$$
Then, the vector $(e,\mathbf{d})$ is indivisible if and only if $l=m$.
\end{Lem}
\begin{Rem}\label{Rem-Betti-Dol-indiv}
(1). Beware that the statement is not about $(e,\mathbf{d}')$. (2). The conclusion applies equally to the morphisms $\mathbf{Iso}_3$ and $\mathbf{NH}_2$, in which case $(e,\mathbf{d})$ is indivisible if and only if $(e,\tilde{\mathbf{d}})$ indivisible. See Remark \ref{Rem-d=tilde-d} (2) below.
\end{Rem}
\begin{proof}
Suppose that $l<m$ and write $m=lm'$. Note that we have tacitly identified the strict $\mathbf{d}\in\mathbb{Z}_{\ge0}^{Q_0}$ with a rank $n$ vector in $\mathbb{Z}_{\ge0}^{\bs\tau(\underline{\nu})}$. Obviously, $\mathbf{d}$ is divisible by $m'$. We have
$$
1=\mathbf{q}^{l\bs\gamma}=\prod_{[j,i]\in\bs\tau(\underline{\nu})}\xi_{[j,i]}^{l\gamma^{\ast}_{[j,i]}}.
$$
If the argument of $\xi_{[j,i]}$ is $-2\pi ia_{[j,i]}$, then $\sum_{[j,i]} a_{[j,i]}l\gamma^{\ast}_{[j,i]}$ is an interger. Under the nonabelian Hodge correspondence, we have 
$$
m'\sum_{[j,i]\in\bs\tau(\underline{\nu})} a_{[j,i]}l\gamma^{\ast}_{[j,i]}=\sum_{[j,i]\in\bs\tau(\underline{\nu})} a_{[j,i]}d^{\ast}_{[j,i]}=\sum_{[j,i]\in\bs\tau(\underline{\mu})} \alpha'_{[j,i]}d^{\prime\ast}_{[j,i]},
$$ 
since each summand $\alpha'_{[j,i]}d^{\prime\ast}_{[j,i]}$ is a sum of those $a_{[j,i']}d^{\ast}_{[j,i']}$ with $a_{[j,i']}=\alpha'_{[j,i]}$. It follows that $e$ is divisible by $m'$, and thus $(e,\mathbf{d})$ is divisible. Conversely, suppose that $(e,\mathbf{d})=m'(e',l\bs\gamma)$ for some $m'>1$. Then, $\sum_{[j,i]} a_{[j,i]}l\gamma^{\ast}_{[j,i]}$ is an integer, which implies $\mathbf{q}^{l\bs\gamma}=1$.
\end{proof}

\subsection{Variation of parabolic weights}\label{subsec-VPW}\hfill

Variation of parabolic weights alters the moduli spaces of parabolic Higgs bundles. This operation has been previously studied by Boden-Hu \cite{BH} (without Higgs fields) and Thaddeus \cite{T02} as a special case of variation of stability conditions, and the focus there was placed on the topology and geometry of the moduli spaces. In our context, it is a tool for reducing the Deligne-Simpson problem to easier cases.

We will define a morphism between two moduli spaces of parabolic Higgs bundles:
$$
\mathbf{Var}_2:\mathcal{M}_{Dol}(\tilde{\mathbf{d}},\bs\alpha,(\mathcal{O}_j)_j)\longrightarrow\mathcal{M}_{Dol}(\mathbf{d}',\bs\alpha',(\mathcal{O}'_j)_j).
$$
The types of $\mathcal{M}_{Dol}(\tilde{\mathbf{d}},\bs\alpha,(\mathcal{O}_j)_j)$ and $\mathcal{M}_{Dol}(\mathbf{d}',\bs\alpha',(\mathcal{O}'_j)_j)$ are $\bs\tau(\underline{\nu})$ and $\bs\tau(\underline{\mu})$ respectively, with each $\mu_j\le\nu_j$. We require that the following conditions are satisfied: 
\begin{itemize}
\item There is a degeneration of types $\sigma:\bs\tau(\underline{\mu})\longrightarrow\bs\tau(\underline{\nu})$ such that $\mathbf{d}'=\sigma^{\ast}\tilde{\mathbf{d}}$.
\item The graded residue $\mathcal{O}_j$ is central in $\bigoplus_{i=0}^{\nu_j}\gl_{d^{\ast}_{[j,i]}}$ with $\nu_j+1$ distinct eigenvalues. Since the flag $E_{[j,\bullet]}$ at $p_j$ has length $\nu_j+1$, it consists precisely of sums of eigenspaces of $\Phi_j$.
\item The graded residue $\mathcal{O}'_j$ is the orbit containing $\mathcal{O}_j$ under the inclusion (see (\ref{eq-asso-d}))
$$
\bigoplus_{i=0}^{\nu_j}\gl_{d^{\ast}_{[j,i]}}\longrightarrow\bigoplus_{i=0}^{\mu_j}\gl_{d^{\prime\ast}_{[j,i]}}.
$$
\end{itemize}
The morphism $\mathbf{Var}_2$ will be defined as a degeneration of parabolic structures:
$$
(E,E_{[\bullet,\bullet]},\Phi)\longmapsto(E,E'_{[\bullet,\bullet]},\Phi)
$$
where $E'_{[j,\bullet]}$ is a degeneration of the flag $E_{[j,\bullet]}$. In view of the above conditions and the fact that $\mathcal{M}_{Dol}(\mathbf{d}',\bs\alpha',(\mathcal{O}'_j)_j)$ is completely determined by the Betti moduli space $\mathcal{M}_B(\mathbf{d},\bs\beta,\bs\xi)$ that we start with, the only flexibility we have is in $\bs\alpha$. 
\begin{Rem}\label{Rem-d=tilde-d}
(1). The filtered local systems in $\mathcal{M}_B(\mathbf{d},\bs\beta,\bs\xi)$ also have type $\bs\tau(\underline{\nu})$. (2). The vectors $\mathbf{d}^{\ast}$ and $\tilde{\mathbf{d}}^{\ast}$ are equal up to permuting their components with the same $j$.
\end{Rem}

For an arbitrary parabolic weight $\bs\alpha$, a degeneration of parabolic structure may not preserve semi-stability. The usual option is to take a generic $\bs\alpha$ in a small neighbourhood of $\bs\alpha'$; however, generic weights do not exist if the dimension vector is divisible. 
\begin{Defn}\label{defn-almost-gen-w}
Fix  $e\in\mathbb{Z}$ and strict $\mathbf{c}\in\mathbb{Z}_{\ge0}^{\bs\tau(\underline{\nu})}$ of rank $n$, and let $\bs\alpha\in\mathbb{R}^{\bs\tau(\underline{\nu})}$ be a parabolic weight such that 
$$
e+\sum_{[j,i]\in\bs\tau(\underline{\nu})} \alpha_{[j,i]}c^{\ast}_{[j,i]}=0.
$$
We say that $\bs\alpha$ is almost generic if for any strict $\mathbf{c}'\in\mathbb{Z}_{\ge0}^{\bs\tau(\underline{\nu})}$ of rank $n'\le n$ and any $e'\in\mathbb{Z}$ the equality
$$
e'+\sum_{[j,i]\in\bs\tau(\underline{\nu})} \alpha_{[j,i]}c^{\prime\ast}_{[j,i]}=0
$$
holds if and only if the vector $(e,(c_{[j,i]})_{[j,i]})$ is a $\mathbb{Q}$-multiple of $(e',(c'_{[j,i]})_{[j,i]})$ (if and only if $(e,(c^{\ast}_{[j,i]})_{[j,i]})$ is a $\mathbb{Q}$-multiple of $(e',(c^{\prime\ast}_{[j,i]})_{[j,i]})$); we say that $\bs\alpha$ is generic if the above equality does not hold unless $(e',(c'_{[j,i]})_{[j,i]})$ is zero or is equal to $(e,(c_{[j,i]})_{[j,i]})$.
\end{Defn}

Let $I=\{x\in\mathbb{R}|0\le x<1\}$ and define the space of weights:
$$
I^{\bs\tau(\underline{\nu})}_{\Delta}:=\{\bs\alpha\in I^{\bs\tau(\underline{\nu})}\mid0\le \alpha_{[j,0]}\le \alpha_{[j,1]}\le\cdots\le\alpha_{[j,\nu_j]}\text{ for any $j$}\}.
$$
Note that we allow $\alpha_{[j,i]}=\alpha_{[j,i+1]}$ in this space. For $(e,\mathbf{c})$ as in the above definition, define
$$
I^{\bs\tau(\underline{\nu})}(e,\mathbf{c}):=\{\bs\alpha\in I^{\bs\tau(\underline{\nu})}_{\Delta}\mid e+\sum_{[j,i]\in\bs\tau(\underline{\nu})} \alpha_{[j,i]}c^{\ast}_{[j,i]}=0\},
$$
which is a hyperplane in $I^{\bs\tau(\underline{\nu})}_{\Delta}$. If a vector $(e',\mathbf{c}')$ is not a scalar multiple of $(e,\mathbf{c})$, then $I^{\bs\tau(\underline{\nu})}(e',\mathbf{c}')$ either does not meet $I^{\bs\tau(\underline{\nu})}(e,\mathbf{c})$ or intersects with $I^{\bs\tau(\underline{\nu})}(e,\mathbf{c})$ in a lower dimensional affine space, thus forming a \textit{wall} in $I^{\bs\tau(\underline{\nu})}(e,\mathbf{c})$. By definition, a weight $\bs\alpha\in I^{\bs\tau(\underline{\nu})}(e,\mathbf{c})$ is almost generic if it lies in the complement of the union of these walls.

Given a degeneration of types $\sigma:\bs\tau(\underline{\mu})\rightarrow\bs\tau(\underline{\nu})$ and a weight $\bs\alpha\in I^{\bs\tau(\underline{\mu})}_{\Delta}$, we define $\smash{\sigma_{\ast}\bs\alpha=((\sigma_{\ast}\alpha)_{[j,i]})_{[j,i]}\in I^{\bs\tau(\underline{\nu})}_{\Delta}}$ by defining $(\sigma_{\ast}\alpha)_{[j,i']}=\alpha_{[j,i]}$ whenever $\sigma(i)\le i'<\sigma(i+1)$. It is easy to see that $\sigma_{\ast}$ restricts to a map
$$
\sigma_{\ast}:I^{\bs\tau(\underline{\mu})}(e,\sigma^{\ast}\mathbf{c})\longrightarrow I^{\bs\tau(\underline{\nu})}(e,\mathbf{c}).
$$
Suppose that $\mathscr{E}=(E,E_{[\bullet,\bullet]},\Phi)$ is a parabolic Higgs bundle of type $\bs\tau(\underline{\nu})$. Then, a parabolic Higgs bundle $\mathscr{E}'=(E,E'_{[\bullet,\bullet]},\Phi)$ of type $\bs\tau(\underline{\mu})$ is obtained from the former by the degeneration $\sigma$ if $E'_{[j,i]}=E_{[j,\sigma(i)]}$ for any $[j,i]$.
\begin{Lem}\label{Lem-para-deg-degen}
Let $\mathscr{E}$ and $\mathscr{E}'$ be as above, and let $\bs\alpha'\in I_{\Delta}^{\bs\tau(\underline{\mu})}$. Then, we have
$$
\deg_{\bs\alpha'}\mathscr{E}'=\deg_{\sigma_{\ast}\bs\alpha'}\mathscr{E}.
$$
\end{Lem}
\begin{proof}
This is simply the equality 
$$
\alpha'_{[j,i']}\dim E'_{[j,i']}/E'_{[j,i'+1]}=\alpha'_{[j,i']}\sum_{\sigma(i')\le i<\sigma(i'+1)}\dim E_{[j,i]}/E_{[j,i+1]}.
$$
\end{proof}
\begin{Lem}\label{Lem-para-type-degen}
Suppose that $(E,E_{[\bullet,\bullet]},\Phi)$ is a parabolic Higgs bundle of type $\bs\tau(\underline{\nu})$ and that $(E,E'_{[\bullet,\bullet]},\Phi)$ is a parabolic Higgs bundle of type $\bs\tau(\underline{\mu})$ obtained from the former by the degeneration $\sigma$. Let $\bs\alpha'$ be a parabolic weight for $E'_{[\bullet,\bullet]}$. Let $F\subset E$ be a $\Phi$-invariant vector bundle, and let $F'_{[\bullet,\bullet]}$ and $F_{[\bullet,\bullet]}$ be the parabolic structures induced by $E'_{[\bullet,\bullet]}$ and $E_{[\bullet,\bullet]}$ respectively. Then, 
\begin{itemize}
\item[(i)] $(F,F'_{[\bullet,\bullet]},\Phi)$ is obtained from $(F,F_{[\bullet,\bullet]},\Phi)$ by a degeneration of types.
\end{itemize}
Moreover, if we denote by $\sigma^F$ the degeneration of types as in (i) and by $\bs\alpha^{\prime F}$ the induced parabolic weight for $F'_{[\bullet,\bullet]}$, then we have
\begin{itemize}
\item[(ii)] $\sigma^{F}_{\ast}\bs\alpha^{\prime F}=\sigma_{\ast}\bs\alpha'$,
\end{itemize}
where the righthand side is the weight for $F_{[\bullet,\bullet]}$ induced from $E_{[\bullet,\bullet]}$.
\end{Lem}
\begin{proof}
(i). The flag $F_{[j,\bullet]}$ on the fiber $F_{p_j}$ consists of subspaces of the form $F_{p_j}\cap E_{[j,i]}$, and similarly for $F'_{[j,\bullet]}$. Since $\{E'_{[j,i']}\mid 0\le i'\le \mu_j\}$ is a subset of $\{E_{[j,i]}\mid 0\le i\le \nu_j\}$, the assertion follows.

(ii). By definition, the subquotient $F'_{[j,i']}/F'_{[j,i'+1]}$ has weight $\alpha'_{[j,i_1]}$ if $E'_{[j,i_1]}$ is the smallest space in the flag $E'_{[j,\bullet]}$ such that the intersection with $F_{p_j}$ is $F'_{[j,i']}$. Then, for $\sigma_F(i')\le i<\sigma_F(i'+1)$, the left hand side of (ii) assigns to the subquotient $F_{[j,i]}/F_{[j,i+1]}$ the weight $\alpha'_{[j,i_1]}$. We need to show that this is also the weight induced from $\sigma_{\ast}\bs\alpha'$. Let $E_{[j,i_2]}$ be the smallest space in the flag $E_{[j,\bullet]}$ such that the intersection with $F_{p_j}$ is $F_{[j,i]}$. We locate $E_{[j,i_2]}$ in $E_{[j,\bullet]}$ as follows. On the one hand, since $F_{[j,i]}\subset F'_{[j,i']}$, we have $E_{[j,i_2]}\subset E'_{[j,i_1]}$. On the other hand, since $F_{[j,i]}\supsetneq F'_{[j,i'+1]}$, we have $E_{[j,i_2]}\supsetneq E'_{[j,i_1+1]}$. By the definition of $\sigma_{\ast}\bs\alpha'$, the induced weight for $F_{[j,i]}/F_{[j,i+1]}$ is also $\alpha'_{[j,i_1]}$.
\end{proof}

Now, let $\mathcal{M}_{Dol}(\tilde{\mathbf{d}},\bs\alpha,(\mathcal{O}_j)_j)$ and $\mathcal{M}_{Dol}(\mathbf{d}',\bs\alpha',(\mathcal{O}'_j)_j)$ be as at the beginning of this subsection, and let $\sigma:\bs\tau(\underline{\mu})\longrightarrow\bs\tau(\underline{\nu})$ be the given degeneration. Let $e$ be determined by $\bs\alpha'$ and $\mathbf{d}'$ as in Lemma \ref{Lem-Betti-Dol-indiv}.

\begin{Prop}\label{Prop-Var2}
Suppose that $\bs\alpha\in I^{\bs\tau(\underline{\nu})}(e,\tilde{\mathbf{d}})$ is such that the interval
\begin{equation}\label{eq-Prop-Var2-0}
\{\sigma_{\ast}\bs\alpha'+t(\bs\alpha-\sigma_{\ast}\bs\alpha')\mid 0< t\le 1\}
\end{equation}
does not meet any wall in $I^{\bs\tau(\underline{\nu})}(e,\tilde{\mathbf{d}})$, and that $\bs\alpha$ is almost generic. Then, the following assertions hold:
\begin{itemize}
\item[(i)] The morphism $\mathbf{Var}_2$ of degenerating parabolic structures is well-defined.
\item[(ii)] The inverse image of an $\bs\alpha'$-stable under $\mathbf{Var}_2$ is nonempty and consists of an $\bs\alpha$-stable parabolic Higgs bundle.
\end{itemize}
\end{Prop} 
\begin{proof}
For (i), we need to show that if $\mathscr{E}:=(E,E_{[\bullet,\bullet]},\Phi)$ is an $\bs\alpha$-semi-stable parabolic Higgs bundle, then $\mathbf{Var}_2(\mathscr{E})=\mathscr{E}':=(E,E'_{[\bullet,\bullet]},\Phi)$ is $\bs\alpha'$-semi-stable. Assume on the contrary that there exists a $\Phi$-invariant subbundle $F\subset E$ with the parabolic structure $F'_{[\bullet,\bullet]}$ induced by $E'_{[\bullet,\bullet]}$ and the induced Higgs field $\Phi$ such that the parabolic degree of $\mathscr{F}':=(F,F'_{[\bullet,\bullet]},\Phi)$ satisfies $\deg_{\bs\alpha'}\mathscr{F}'>0$ (recall that $\bs\alpha'$ naturally induces a weight for $F'_{[\bullet,\bullet]}$). Let $\mathscr{F}:=(F,F_{[\bullet,\bullet]},\Phi)$ be the parabolic Higgs bundle with the parabolic structure induced from $\mathscr{E}$ and denote by $\mathbf{d}^F$ the dimension vector of its parabolic structures. Consider the parabolic weight $\sigma_{\ast}\bs\alpha'$ for $\mathscr{F}$ induced from $\mathscr{E}$. Then, we have 
\begin{equation}\label{eq-Prop-Var2-1}
\deg_{\sigma_{\ast}\bs\alpha'}\mathscr{F}>0
\end{equation}
by Lemma \ref{Lem-para-deg-degen} and Lemma \ref{Lem-para-type-degen}. However, $\mathscr{E}$ is $\bs\alpha$-semi-stable, and thus $\deg_{\bs\alpha}\mathscr{F}\le 0$. Since $\bs\alpha$ is almost generic, we have either 
\begin{itemize}
\item $\deg_{\bs\alpha}\mathscr{F}<0$, or
\item $\deg_{\bs\alpha}\mathscr{F}=0$, in which case $(\deg F,\mathbf{d}^F)$ is a scalar multiple of $(e,\tilde{\mathbf{d}})$. (We regard $\mathbf{d}^F$ as an element of $\mathbb{Z}_{\ge0}^{\bs\tau(\underline{\nu})}$ by defining $d^F_{[j,i]}=\dim E_{[j,i]}\cap F_{p_j}$ for any $[j,i]\in\bs\tau(\underline{\nu})$.)
\end{itemize}
Should the second case occur, the equality $\deg_{\sigma_{\ast}\bs\alpha'}\mathscr{E}=\deg_{\bs\alpha'}\mathscr{E}'=0$ would imply $\deg_{\sigma_{\ast}\bs\alpha'}\mathscr{F}=0$, which contradicts (\ref{eq-Prop-Var2-1}). However, the first case $\deg_{\bs\alpha}\mathscr{F}<0$ implies that the interval (\ref{eq-Prop-Var2-0}) meets a wall, contradicting to our assumption. It follows that $\mathbf{Var}_2(\mathscr{E})$ is $\bs\alpha'$-semi-stable.

Now we prove (ii). Observe that for any $\mathscr{E}'=(E,E'_{[\bullet,\bullet]},\Phi)\in \mathcal{M}(\mathbf{d}',\bs\alpha',(\mathcal{O}'_j)_j)$, there is a (unique) refinement $E_{[\bullet,\bullet]}$ of the flag $E'_{[\bullet,\bullet]}$ such that the graded residue of $\Phi$ at each $p_j$ lies in the adjoint orbit $\mathcal{O}_j$.  We need to show that if $\mathscr{E}'$ is $\bs\alpha'$-stable, then $\mathscr{E}$ is $\bs\alpha$-stable; in particular, $\mathscr{E}$ lies in $\mathbf{Var}_2^{-1}(\mathscr{E}')$. Assume on the contrary that $\mathscr{E}$ is not $\bs\alpha$-stable, so that there exists a vector subbundle $F\subset E$ that defines a parabolic Higgs subbundle $\mathscr{F}$ with $\deg_{\bs\alpha}\mathscr{F}\ge0$. However, the parabolic Higgs subbundle $\mathscr{F}'$ satisfies $\deg_{\bs\alpha'}\mathscr{F}'<0$. By Lemma \ref{Lem-para-deg-degen} and Lemma \ref{Lem-para-type-degen} again, we have $\deg_{\sigma_{\ast}\bs\alpha'}\mathscr{F}<0$. Similar arguments as in the previous paragraph show that this is not possible.
\end{proof}

\begin{Prop}\label{Prop-Var2-inf}
Suppose that we are in case ($\mathbf{Aff}^{\infty}$) and that $\bs\alpha$ is a generic parabolic weight such that the interval
\begin{equation}\label{eq-Prop-Var2-inf-0}
\{\sigma_{\ast}\bs\alpha'+t(\bs\alpha-\sigma_{\ast}\bs\alpha')\mid 0< t\le 1\}
\end{equation}
does not meet any wall in $I^{\bs\tau(\underline{\nu})}(e,\tilde{\mathbf{d}})$. Then, the morphism $\mathbf{Var}_2$ is well-defined and induces a bijection of $\mathbb{C}$-points. 
\end{Prop}
\begin{proof}
In case ($\mathbf{Aff}^{\infty}$), the vector $\mathbf{d}$ is divisible, and so is $\tilde{\mathbf{d}}$ by Remark \ref{Rem-d=tilde-d} (2). We can therefore choose $\bs\alpha$ to be generic, and thus Proposition \ref{Prop-Var2} applies, giving a well-defined $\mathbf{Var}_2$. A simplification in this case is that 
$\mathcal{M}_{Dol}(\mathbf{d}',\bs\alpha',(\mathcal{O}'_j)_j)$ consists of stable parabolic Higgs bundles (because the stability condition on the Betti side is chosen to be generic), although $\bs\alpha'$ is not necessarily generic. The second part of Proposition \ref{Prop-Var2} shows that $\mathbf{Var}_2$ is bijective, since every object is stable in these moduli spaces.
\end{proof}
\begin{Rem}
We will see in the next subsection that if $\bs\alpha$ is generic, then the domain of $\mathbf{Var}_2$ is connected. If its target is also normal, then it follows from the Zariski main theorem that $\mathbf{Var}_2$ is an isomorphism. Normality is a reasonable assumption but does not seem to have been established in the literature: On the one hand, the moduli of strongly parabolic Higgs bundles are known to be normal; on the other hand, the normality of multiplicative quiver varieties can be transferred to the Dolbeault side if Simpson's isosingularity theorem \cite[Theorem 10.6]{Si94b} holds in this generality.
\end{Rem}

\subsection{Forgetting filtered structures}\label{subsec-FFS}\hfill

Let us clarify the last step in our reduction process: the isomorphism
$$
\mathbf{Iso}_3:\mathcal{M}_B(\mathbf{d},\bs\beta,\bs\xi')\lisom\mathcal{M}_B(\mathcal{C}'),
$$
which is defined by forgetting the filtered structures. The space $\mathcal{M}_B(\mathbf{d},\bs\beta,\bs\xi')$ parametrises local systems of filtered type $\mathbf{d}$ and weights $\bs\beta$, whose graded monodromy at each puncture is semi-simple with eigenvalues $\xi'_{[j,i]}$. Write $\mathcal{C}'=(C'_j)_{1\le j\le k}$. Then, the conjugacy class $C'_j$ is semi-simple and has eigenvalue $\xi'_{[j,i]}$ with multiplicity $d^{\ast}_{[j,i]}$. Forgetting the filtered structure obviously defines a morphism $\mathbf{Iso}_3$; in terms of multiplicative quiver varieties, this is simply changing the stability condition into the trivial one. We need to show that it is an isomorphism.

\begin{Lem}\label{Lem-almost-Dol-to-Betti}
Consider the nonabelian Hodge correspondence 
$$
\mathbf{NH}_2:\mathcal{M}_B(\mathbf{d},\bs\beta,\bs\xi') \lisom \mathcal{M}_{Dol}(\tilde{\mathbf{d}},\bs\alpha,(\mathcal{O}_j)_j).
$$
Suppose that $\bs\alpha$ is an almost generic parabolic weight. Then, the eigenvalues $\bs\xi'$ are almost generic, and thus the conjugacy classes $\mathcal{C}'$ are almost generic. In particular, if $\bs\alpha$ is generic, then the eigenvalues $\bs\xi'$ and thus $\mathcal{C}'$ are generic.
\end{Lem}
\begin{proof}
Let $\mathbf{q}'$ be the deformation parameter defined by $\bs\xi'$ as in \S \ref{subset-MQV} and let $\bs\gamma\le \mathbf{d}$ be a strict dimension vector. Suppose that $\mathbf{q}^{\prime\bs\gamma}=1$. We need to show that $\mathbf{d}$ is a multiple of $\bs\gamma$. Let the argument of $\xi_{[j,i]}$ be $-2\pi ia_{[j,i]}$, the equality $\mathbf{q}^{\prime\bs\gamma}=1$ implies that
$$
\sum_{[j,i]\in\bs\tau(\underline{\nu})}a_{[j,i]}\gamma^{\ast}_{[j,i]}=-e'
$$
for some $e'\in\mathbb{Z}$, where we regard $\bs\gamma$ as an element of $\mathbb{Z}_{\ge0}^{\bs\tau(\underline{\nu})}$ with $\gamma_{[j,0]}=\gamma_{\star}$ for any $j$. Up to a permutation of the $i$-indices, the numbers $a_{[j,\bullet]}$ are precisely the weights $\alpha_{[j,\bullet]}$. Let $\tilde{\bs\gamma}^{\ast}$ be the vector obtained from $\bs\gamma^{\ast}$ by permuting the $i$-indices in the same way. Since the weight $\bs\alpha$ is almost generic, we have $m'(e',\tilde{\bs\gamma}^{\ast})=(e,\tilde{\mathbf{d}}^{\ast})$ for some $m'\in\mathbb{Z}_{>0}$; therefore, $m'\bs\gamma=\mathbf{d}$. 
\end{proof}

\begin{Prop}\label{Prop-Iso3}
Suppose that $\bs\xi'$ is almost generic. Then, the morphism
$$
\mathbf{Iso}_3:\mathcal{M}_B(\mathbf{d},\bs\beta,\bs\xi')\longrightarrow\mathcal{M}_B(\mathcal{C}')
$$
is an isomorphism. Moreover, this isomorphism matches $\bs\beta$-stable filtered local systems with irreducible local systems.
\end{Prop}
\begin{proof}
Identify the Betti moduli spaces with multiplicative quiver varieties $\mathcal{M}(\mathbf{q}',\mathbf{d})$ and $\mathcal{M}_{\bs\theta}(\mathbf{q}',\mathbf{d})$. Suppose that $\mathbf{d}=m\mathbf{d}_0$ with $\mathbf{d}_0$ indivisible. The deformation parameter $\mathbf{q}'$ is almost generic by assumption. We will  show that $\mathbf{Iso}_3$ induces a bijection of $\mathbb{C}$-points. By Theorem \ref{Thm-Norm}, the target $\mathcal{M}(\mathbf{q}',\mathbf{d})$ is normal. Let $X\subset \mathcal{M}(\mathbf{q}',\mathbf{d})$ be a connected component and let $Y=\mathbf{Iso}_3^{-1}(X)$. Let us assume the bijection of points and show that $Y$ is isomorphic to $X$. The bijection implies that there is a connected component $Y_0$ of $Y$ that dominates $X$. However, the morphism $\mathbf{Iso}_3$ is projective since it is defined by degenerating stability conditions (or rather, as forgetting flags); therefore, $\mathbf{Iso}_3(Y_0)=X$ and thus $Y_0=Y$. It follows from the Zariski main theorem that the restriction of $\mathbf{Iso}_3$ to $Y$ is an isomorphism.

Let us show the surjectivity. Let $\rho_1\oplus\cdots\oplus\rho_r$ be a direct sum of simple representations $\rho_s$. Since $\mathbf{q}'$ is almost generic, the dimension vector $\mathbf{d}^{(s)}$ of each $\rho_s$ is a multiple of $\mathbf{d}_0$, and thus $\mathbf{d}^{(s)}\cdot\bs\theta=0$. It follows that $\rho_1\oplus\cdots\oplus\rho_r$ is a $\bs\theta$-polystable representation, whence surjectivity.

Next, we show the injectivity. Let $\rho_1\oplus\cdots\oplus\rho_r$ be as above, and suppose that $\rho'_1\oplus\cdots\oplus\rho'_t$ is a direct sum of $\bs\theta$-stable representations $\rho'_s$, which has $\rho_1\oplus\cdots\oplus\rho_r$ as its semi-simplification. Obviously, we have $t\le r$. Let us show that $t<r$ is not possible. The dimension vector of each $\rho'_s$ is also a multiple of $\mathbf{d}_0$. By replacing $\mathbf{d}$ by a smaller dimension vector if necessary, we may assume $t=1$; that is, the $\bs\theta$-stable representation $\rho'_1$ has $\rho_1\oplus\cdots\oplus\rho_r$ as its semi-simplification. However, we have $\mathbf{d}^{(s)}\cdot\bs\theta=0$ for any $s$, contradicting the stability of $\rho'_1$. We have shown that $t=r$. Then, each $\rho'_s$ is necessarily simple, and thus is isomorphic to some $\rho_{s'}$. The injectivity follows. We have also proved in the process that a $\bs\theta$-stable representation is necessarily simple, which proves the second statement of the proposition.
\end{proof}

\section{Generalisation and application}

\subsection{Dimension vectors of $\bs\theta$-stable representations}\label{subsec-Dim-theta-st}\hfill

The Deligne-Simpson problem can be regarded as the trivial stability case of the more general question about the existence of $\bs\theta$-stable representations. The techniques that we have used so far work in this generality as well. Let $Q$ be a star-shaped quiver defined by some tuple of integers $\underline{\nu}$ as in \S \ref{subsec-Comb-data}. Fix $\mathbf{q}\in(\mathbb{C}^{\ast})^{Q_0}$ and $\bs\theta\in\mathbb{R}^{Q_0}$, define $R_{\mathbf{q},\bs\theta}^+:=\{\mathbf{d}\in R^+|\mathbf{q}^{\mathbf{d}}=1,~\bs\theta\cdot\mathbf{d}=0\}$ and
\begingroup
\allowdisplaybreaks
\begin{align*}
\Sigma_{\mathbf{q},\bs\theta}:=\{\mathbf{d}\in R^+_{\mathbf{q},\bs\theta}\mid&\text{if $\mathbf{d}=\sum_{s=1}^r\mathbf{d}^{(s)}$ with $r\ge2$ and each $\mathbf{d}^{(s)}\in R^+_{\mathbf{q},\bs\theta}$,}\\
&\text{then $p(\mathbf{d})>\sum_{s=1}^rp(\mathbf{d}^{(s)})$}\}.
\end{align*}
\endgroup
For any $v\in Q_0$, define 
\begingroup
\allowdisplaybreaks
\begin{align*}
r_v(\bs\theta)&=(\theta_w-(e_v,e_w)\theta_w)_{w\in Q_0},\text{ and}\\
u_v(\mathbf{q})&=(q_v^{-(e_v,e_w)}q_w)_{w\in Q_0},
\end{align*}
\endgroup
for any $\bs\theta\in\mathbb{R}^{Q_0}$ and $\mathbf{q}\in(\mathbb{C}^{\ast})^{Q_0}$. The map $(\mathbf{q},\mathbf{d},\bs\theta)\mapsto(u_v(\mathbf{q}),s_v(\mathbf{d}),r_v(\bs\theta))$ is called an admissible reflection if either $\theta_v\ne 0$ or $q_v\ne 1$. According to Yamakawa \cite[Theorem 5.3]{Yama}, admissible reflections induce isomorphisms between multiplicative quiver varieties (his proof works beyond the stable loci).
\begin{Thm}\label{Thm-theta-st-dim}
For a star-shaped quiver, the following are equivalent:
\begin{itemize}
\item[(i)] There exists a $\bs\theta$-stable representation in $\mathcal{M}_{\bs\theta}(\mathbf{q},\mathbf{d})$.
\item[(ii)] The dimension vector $\mathbf{d}$ lies in $\Sigma_{\mathbf{q},\bs\theta}$.
\end{itemize}
\end{Thm}
\begin{proof}
 Up to a sequence of admissible reflections, we may assume that $\theta_{[j,i]}\ge0$ for any $j$ and any $i>0$. To obtain such a $\theta$, we may begin by considering the vertex $[j,\nu_j]$ for some $j$. If $\theta_{[j,\nu_j]}<0$, then we apply the reflection $r_{[j,\nu_j]}$. If $\theta_{[j,i]}\ge0$ for $i_1<i\le\nu_j$ and $\theta_{[j,i_1]}<0$, then we may apply a sequence of admissible reflections $r_{[j,i_2]}\cdots r_{[j,i_1+1]} r_{[j,i_1]}$ for some $i_1\le i_2\le \nu_j$ so that the resulting stability parameter has nonnegative value at $[j,i]$ for every $i_1\le i\le\nu_j$. If $d_{[j,i]}>d_{[j,i+1]}$ for any $[j,i]$, then these inequalities  are preserved under reflections; moreover, if $\mathbf{d}$ is indivisible, then it remains so under reflections at $[j,i]$ for $i>0$, since such a reflection permutes $d_{[j,i]}-d_{[j,i+1]}$ and $d_{[j,i-1]}-d_{[d,i]}$.

We first show (i)$\Rightarrow$(ii) in parallel with previous sections. The first step is again Schedler-Tirelli's classification result. Theorem \ref{Thm-dimvec-ST} is in fact the $\bs\theta=0$ version of \cite[Corollary 6.18]{ST}. The full statement says that if (i) is true, then one of the following occurs:
\begin{itemize}
\item $\mathbf{d}\in\Sigma_{\mathbf{q},\bs\theta}$.
\item[($\mathbf{Aff}$)] $\mathbf{d}\in\mathbb{Z}_{\ge2}\cdot\Sigma_{\mathbf{q},\bs\theta}^{iso}$, where $\Sigma_{\mathbf{q},\bs\theta}^{iso}\subset\Sigma_{\mathbf{q},\bs\theta}$ is the subset of isotropic imaginary roots.
\item[($\mathbf{Aff}^{\infty}$)] $\mathbf{d}=e_{\infty}+m\bs\delta$ is a flat root and $m\ge2$. Moreover, we have $q_{\infty}=\mathbf{q}^{\bs\delta}=1$ and $\theta_{0,\infty}=\bs\theta\cdot\bs\delta=0$.
\end{itemize} 
We need to rule out vectors of type ($\mathbf{Aff}$) and ($\mathbf{Aff}^{\infty}$). Note that after adjusting $\bs\theta$ as above, we may not have $d_{\infty}=1$ in case ($\mathbf{Aff}^{\infty}$), but $\mathbf{d}$ is nevertheless indivisible. We assume that there is a $\bs\theta$-stable representation $\rho_s\in\mathcal{M}_{\bs\theta}(\mathbf{q},\mathbf{d})$ in these cases and prove by contradiction. Let $\tilde{\bs\theta}$ be a generic (resp. almost generic) stability condition in a small neighbourhood of $\bs\theta$ in case ($\mathbf{Aff}^{\infty}$) (resp. ($\mathbf{Aff}$)) which also satisfies $\tilde{\theta}_{[j,i]}>0$ as long as $i>0$. Here, we say that $\tilde{\bs\theta}$ is almost generic if the only dimension vectors $\mathbf{d}'\le\mathbf{d}$ satisfying $\tilde{\bs\theta}\cdot\mathbf{d}'=0$ are $\mathbb{Q}$-multiples of $\mathbf{d}$. Then, there is a well-defined morphism
$$
\mathcal{M}_{\bs\theta}(\mathbf{q},\mathbf{d})\stackrel{\mathbf{Var}_0}{\longleftarrow}  \mathcal{M}_{\tilde{\bs\theta}}(\mathbf{q},\mathbf{d}) 
$$
such that $\rho_s$ lifts to a $\tilde{\bs\theta}$-stable representaion $\rho'_s\in\mathcal{M}_{\tilde{\bs\theta}}(\mathbf{q},\mathbf{d})$. The existence of such a morphism is clear if we are in case ($\mathbf{Aff}^{\infty}$) and $\tilde{\bs\theta}$ is generic; in case ($\mathbf{Aff}$), the argument is similar to the proof of Proposition \ref{Prop-Var2}, for which we give some details below. We need to check that (1) every $\tilde{\bs\theta}$-semi-stable representation is $\bs\theta$-semi-stable and (2) every $\bs\theta$-stable representation is $\tilde{\bs\theta}$-stable. Suppose that $\rho$ is $\tilde{\bs\theta}$-semi-stable but not $\bs\theta$-semi-stable. Let $\rho_1\subset\rho$ be a subrepresentation with dimension vector $\mathbf{d}_1$ satisfying $\bs\theta\cdot\mathbf{d}_1>0$. However, we have $\tilde{\bs\theta}\cdot\mathbf{d}_1\le0$. If $\tilde{\bs\theta}\cdot\mathbf{d}_1<0$, then the interval connecting $\bs\theta$ and $\tilde{\bs\theta}$ meets a wall, contradicting the assumption that $\tilde{\bs\theta}$ is (almost) generic and lies in a small neighbourhood of $\bs\theta$. If $\tilde{\bs\theta}\cdot\mathbf{d}_1=0$, then $\mathbf{d}_1$ is a $\mathbb{Q}$-multiple of $\mathbf{d}$ since $\tilde{\bs\theta}$ is almost generic. But this would imply that $\bs\theta\cdot\mathbf{d}_1=0$, contradicting the assumption $\bs\theta\cdot\mathbf{d}_1>0$, and thus (1) follows. The proof of (2) is similar. In case ($\mathbf{Aff}^{\infty}$), the morphism $\mathbf{Var}_0$ is in fact a resolution by \cite[Theorem 6.23]{ST}. The rest of the proof proceeds exactly as in \S \ref{subsec-proof-ThmA}, replacing the dimension formula of \cite[Lemma 7.1]{CBS} by \cite[Proposition 2.15]{ST}.

Now, we prove (ii)$\Rightarrow$(i). The assumption $\mathbf{d}\in\Sigma_{\mathbf{q},\bs\theta}$ says in particular that $\mathbf{d}$ is a positive root. Now, a positive root supported on a star-shaped quiver is either a real root supported on a leg or satisfies $d_{\star}>0$ and $d_{[j,i]}\ge d_{[j,i+1]}$ for any $[j,i]$. If $\mathbf{d}$ is supported on a leg, then the equivalence between (i) and (ii) is clear. We assume $d_{\star}>0$ in what follows. We may further assume $d_{[j,i]}> d_{[j,i+1]}$ for any $[j,i]$ after the following reduction procedure. Suppose that $[j,i_0]$ is such that $d_{[j,i_0-1]}=d_{[j,i_0]}$ and $d_{[j,i_0]}>d_{[j,i_0+1]}$. If $q_{[j,i_0]}\ne 1$ or $\theta_{[j,i_0]}\ne 0$, then we may apply an admissible reflection at $[j,i_0]$ so that the new dimension vector $\mathbf{d}'$ satisfies $d'_{[j,i_0]}=d_{[j,i_0+1]}$. An induction on the maximal possible $d\in\mathbb{Z}$ such that $d=d_{[j,i]}=d_{[d,i+1]}$ for some $[j,i]$ will result in a $\mathbf{d}$ with $d_{[j,i]}> d_{[j,i+1]}$ for all $[j,i]$. If there happen to be some $[j,i_0]$ as above but with $q_{[j,i_0]}=1$ and $\theta_{[j,i_0]}=0$, then it is easy to see that $\mathbf{d}\notin\Sigma_{\mathbf{q},\bs\theta}$ in this case. 

In view of \cite[Proposition 2.19]{ST}, it suffices to show that $\mathcal{M}_{\bs\theta}(\mathbf{q},\mathbf{d})$ is nonempty. The argument in the previous paragraph shows that there is a well-defined morphism $\mathcal{M}_{\tilde{\bs\theta}}(\mathbf{q},\mathbf{d})\rightarrow\mathcal{M}_{\bs\theta}(\mathbf{q},\mathbf{d})$ for some almost generic $\tilde{\bs\theta}$ such that $\tilde{\theta}_{[j,i]}>0$ for all $i>0$ (the construction does not use any particular properties of ($\mathbf{Aff}$) or ($\mathbf{Aff}^{\infty}$)). The variety $\mathcal{M}_{\tilde{\bs\theta}}(\mathbf{q},\mathbf{d})$ is isomorphic to a Betti moduli space $\mathcal{M}_B(\mathbf{d},\bs\beta,\bs\xi)$, where $\bs\beta$ and $\bs\xi$ are determined by $\tilde{\bs\theta}$ and $\mathbf{q}$ respectively. As in (\ref{eq-red-diag}), we consider the following sequence of morphisms
\begin{equation}\label{eq-Thm-theta-st-dim}
\mathcal{M}_B(\mathbf{d},\bs\beta,\bs\xi)\stackrel{\mathbf{NH}_1}{\rightarrow}\mathcal{M}_{Dol}(\mathbf{d}',\bs\alpha',(\mathcal{O}'_j)_j)\stackrel{\mathbf{Var}_2}{\leftarrow} \mathcal{M}_{Dol}(\tilde{\mathbf{d}},\bs\alpha,(\mathcal{O}_j)_j)\stackrel{\mathbf{NH}_2}{\leftarrow}  \mathcal{M}_B(\mathbf{d},\bs\beta,\bs\xi')\stackrel{\mathbf{Iso}_3}{\rightarrow} \mathcal{M}_B(\mathcal{C}').
\end{equation}
As before, $\mathbf{NH}_1$ and $\mathbf{NH}_2$ are nonabelian Hodge correspondences, $\mathbf{Var}_2$ is given by Proposition \ref{Prop-Var2}, and $\mathbf{Iso}_3$ is given by Proposition \ref{Prop-Iso3}. We have $\mathcal{M}_B(\mathcal{C}')\cong\mathcal{M}(\mathbf{q}',\mathbf{d})$ for some almost generic $\mathbf{q}'$ determined by $\bs\xi'$. By assumption, $\mathbf{d}$ is a root and lies in $\Sigma_{\mathbf{q}'}$. It follows from \cite[Theorem 1.3]{CB04} that $\mathcal{M}(\mathbf{q}',\mathbf{d})$ is nonempty; therefore, $\mathcal{M}_{\bs\theta}(\mathbf{q},\mathbf{d})$ is nonempty. This completes the proof of the theorem.
\end{proof}

\subsection{Connectedness of character varieties with nongeneric monodromies}\hfill

As is mentioned in \S \ref{subsec-conn-char}, the connectedness of character varieties with generic monodromies has been established in the literature. We prove below instances of connected character varieties without the genericity assumption. This result will be used in the next subsection to study the decomposition of character varieties. 

\begin{Prop}\label{Prop-Conn-nongen}
Let $\mathcal{M}_{\bs\theta}(\mathbf{q},\mathbf{d})$ be a nonempty multiplicative quiver variety for some star-shaped quiver $Q$. Suppose that $d_{[j,i]}> d_{[j,i+1]}$ for any $j$ and any $i\ge0$, that $\mathbf{d}\in\Sigma_{\mathbf{q},\bs\theta}$ and $\mathbf{d}=m\mathbf{d}_0$ for some indivisible dimension vector $\mathbf{d}_0$, and that $\mathbf{q}^{\mathbf{d}_0}$ is an $m$-th primitive root of unity. Then, the variety $\mathcal{M}_{\bs\theta}(\mathbf{q},\mathbf{d})$ is connected.
\end{Prop}
\begin{Rem}
The most general case allows $\mathbf{q}^{\mathbf{d}_0}$ to be an $l$-th primitive root of unity with $l<m$.
\end{Rem}
\begin{proof}
As in the proof of Theorem \ref{Thm-theta-st-dim}, we may assume that $\theta_{[j,i]}\ge0$ for any $j$ and any $i>0$. The vector $\mathbf{d}$ as in the statement of the proposition is called $q$-indivisible in \cite{ST}. By \cite[Theorem 6.23]{ST}, for an almost generic $\tilde{\bs\theta}$ in a small neighbourhood of $\bs\theta$, variation of stability defines a resolution $\mathbf{Var}_0:\mathcal{M}_{\tilde{\bs\theta}}(\mathbf{q},\mathbf{d})\longrightarrow  \mathcal{M}_{\bs\theta}(\mathbf{q},\mathbf{d})$. The stability condition $\tilde{\bs\theta}$ can be chosen in such a way that $\tilde{\theta}_{[j,i]}>0$ as long as $i>0$. Again, we use the morphisms (\ref{eq-Thm-theta-st-dim}) to show that $\mathcal{M}_{\tilde{\bs\theta}}(\mathbf{q},\mathbf{d})$ is connected. Since $\mathcal{M}_{\tilde{\bs\theta}}(\mathbf{q},\mathbf{d})$ consists of stable representations, the space $\mathcal{M}_{Dol}(\mathbf{d}',\bs\alpha',(\mathcal{O}'_j)_j)$ consists of stable parabolic Higgs bundles. We may choose $\bs\alpha$ to be generic so that $\mathcal{M}_{Dol}(\tilde{\mathbf{d}},\bs\alpha,(\mathcal{O}_j)_j)$ also consists of stable parabolic Higgs bundles. Indeed, our assumptions on $\mathbf{d}$ and $\mathbf{q}$ imply that $(e,\tilde{\mathbf{d}})$ is indivisible by Lemma \ref{Lem-Betti-Dol-indiv}, where $e$ is the degree of the underlying vector bundles in these moduli spaces. By Proposition \ref{Prop-Var2}, $\mathbf{Var}_2$ is bijective. Now, the eigenvalues $\bs\xi'$ are generic by Lemma \ref{Lem-almost-Dol-to-Betti} and our choice of $\bs\alpha$; therefore, $\mathbf{Iso}_3$ is an isomorphism by Proposition \ref{Prop-Iso3}. Finally, the morphisms (\ref{eq-Thm-theta-st-dim}) combined with Theorem \ref{Thm-Conn-gene} show that $\mathcal{M}_{\tilde{\bs\theta}}(\mathbf{q},\mathbf{d})$ is connected.
\end{proof}

\subsection{Decomposition of character varieties}\hfill

We will prove in this subsection the multiplicative counterpart of Crawley-Boevey's decomposition of additive quiver varieties \cite{CB02}, but only for those that are isomorphic to character varieties for $\mathbb{P}^1$; this also refines the decomposition proved by Schedler-Tirelli (see \cite[\S 6.3]{ST}).

\begin{Thm}\label{Thm-Sym-iso}
Let $\mathbf{d}\in\Sigma_{\mathbf{q},\bs\theta}$ be an isotropic imaginary root. Then, taking direct sums induces an isomorphism
$$
\psi:S^{m'}\mathcal{M}_{\bs\theta}(\mathbf{q},\mathbf{d})\lisom\mathcal{M}_{\bs\theta}(\mathbf{q},m'\mathbf{d}),
$$
where $S^{m'}\mathcal{M}_{\bs\theta}(\mathbf{q},\mathbf{d})$ is the symmetric product of $\mathcal{M}_{\bs\theta}(\mathbf{q},\mathbf{d})$. In particular, $\mathcal{M}_{\bs\theta}(\mathbf{q},m'\mathbf{d})$ is connected.
\end{Thm}
\begin{proof}
Write $\mathbf{d}=l\bs\delta$, where $\bs\delta$ is the minimal positive imaginary root of the supporting affine Dynkin diagram. The condition $\mathbf{d}\in\Sigma_{\mathbf{q},\bs\theta}$ implies that $\mathbf{q}^{\bs\delta}$ is a primitive $l$-th root of unity. We know that $\mathcal{M}_{\bs\theta}(\mathbf{q},\mathbf{d})$ and thus $S^{m'}\mathcal{M}_{\bs\theta}(\mathbf{q},\mathbf{d})$ are connected by Proposition \ref{Prop-Conn-nongen}. By \cite[Theorem 5.4]{KS}, the variety $\mathcal{M}_{\bs\theta}(\mathbf{q},m'\mathbf{d})$ is normal. It remains to show that taking direct sums induces a bijection of $\mathbb{C}$-points.

Surjectivity follows from \cite[Proposition 7.2]{ST}. We give some details for completeness. Let $\rho=\bigoplus_{s=1}^r\rho_s\in\mathcal{M}_{\bs\theta}(\mathbf{q},m'\mathbf{d})$ be a direct sum of $\bs\theta$-stable representations $\rho_s$. Theorem \ref{Thm-theta-st-dim} implies that the dimension vector of each $\rho_s$ lies in $\Sigma_{\mathbf{q},\bs\theta}$; that is, we have a decomposition of the dimension vector $m'\mathbf{d}=\sum_s\mathbf{d}^{(s)}$ with each $\mathbf{d}^{(s)}\in\Sigma_{\mathbf{q},\bs\theta}$. By Lemma \ref{Lem-min-admi-m'ldelta} below, this decomposition is a refinement of $\sum_{t}\mathbf{d}^{(t)}$ where each $\mathbf{d}^{(t)}=\mathbf{d}$. It follows that we can rewrite $\bigoplus_s\rho_s$ as $\bigoplus_{t=1}^{m'}\bigoplus_{s\in \Lambda_t}\rho_s$ where $\{\Lambda_t\}_t$ is a partition of the set $\{1,2,\ldots,r\}$ and the dimension vector of $\bigoplus_{s\in \Lambda_t}\rho_s$ is $\mathbf{d}$. This proves the surjectivity. In particular, $\mathcal{M}_{\bs\theta}(\mathbf{q},m'\mathbf{d})$ is connected and thus irreducible.

Let us show that this morphism between irreducible normal varieties is birational. By \cite[Proposition 2.15]{ST}, which is the $\bs\theta$-version of \cite[Lemma 7.1]{CBS}, the stratum of $\mathcal{M}_{\bs\theta}(\mathbf{q},m'\mathbf{d})$ consisting of mutually nonisomorphic $\mathbf{d}$-dimension $\bs\theta$-stable representations is open (and has dimension $2m'$). The open subset of $S^{m'}\mathcal{M}_{\bs\theta}(\mathbf{q},\mathbf{d})$ consisting of mutually nonisomorphic $\mathbf{d}$-dimension $\bs\theta$-stable representations is in bijection with this open stratum of $\mathcal{M}_{\bs\theta}(\mathbf{q},m'\mathbf{d})$; therefore, they are isomorphic by Zariski main theorem, hence birationality.

Obviously, the fibres of $\psi$ are finite. A version of Zariski main theorem applies to this situation (quasi-finite birational morphism) and implies that the surjective morphism $\psi$ is an open immersion and thus an isomorphism.
\end{proof}

Let $\mathbf{d}\in \mathbb{Z}_{\ge0}^{Q_0}$. Let $\mathbf{d}=\sum_{s=1}^{r_1}\mathbf{d}^{(s)}$ and $\mathbf{d}=\sum_{t=1}^{r_2}\mathbf{d}^{(t)}$ be two decompositions of $\mathbf{d}$. We say that the former is a refinement of the latter if there is a partition $\{\Lambda_t\}_t$ of the set $\{1,2,\ldots,r_1\}$ such that $\mathbf{d}^{(t)}=\sum_{s\in\Lambda_t}\mathbf{d}^{(s)}$ for all $t$. We say that a decomposition $\mathbf{d}=\sum_{t=1}^{r}\mathbf{d}^{(t)}$ is a $\Sigma_{\mathbf{q},\bs\theta}$-decomposition if each $\mathbf{d}^{(t)}$ lies in $\Sigma_{\mathbf{q},\bs\theta}$. We say that a $\Sigma_{\mathbf{q},\bs\theta}$-decomposition $\mathbf{d}=\sum_{t=1}^{r}\mathbf{d}^{(t)}$ is minimal if any other $\Sigma_{\mathbf{q},\bs\theta}$-decomposition $\mathbf{d}=\sum_{s=1}^{r'}\mathbf{d}^{(s)}$ is a refinement of it.
\begin{Lem}\label{Lem-min-admi-m'ldelta}
Suppose that $\bs\delta$ is the minimal positive imaginary root supported on an affine Dynkin quiver $Q$. Let $\mathbf{q}\in(\mathbb{C}^{\ast})^{Q_0}$ be such that $\mathbf{q}^{\bs\delta}$ is a primitive $l$-th root of unity, and let $\bs\theta$ be such that $\bs\theta\cdot\bs\delta=0$. Then, for any $m'\in\mathbb{Z}_{>0}$, the minimal $\Sigma_{\mathbf{q},\bs\theta}$-decomposition of $m'l\bs\delta$ is $l\bs\delta+\cdots+l\bs\delta$.
\end{Lem}
\begin{proof}
This proof is parallel to \cite[Lemma 3.2]{CB02}; we only indicate the key steps. An essential fact used in the proof is that an affine Kac-Moody root system consists of the vectors $\mathbf{d}_1+m\bs\delta$, where $m\in\mathbb{Z}$ and $\mathbf{d}_1$ is a root of the corresponding finite type root system or is zero if $m\ne 0$. This allows us to show that if $\mathbf{d}\in\Sigma_{\mathbf{q},\bs\theta}\setminus\{l\bs\delta\}$, then $\mathbf{d}$ is a real root and $\mathbf{d}< l\bs\delta$ (i.e., $d_v\le l\delta_v$ for every $v\in Q_0$ and $\mathbf{d}\ne l\bs\delta$). As in the proof of \cite[Lemma 3.2]{CB02}, if $\mathbf{d}=\sum_{s=1}^r\mathbf{d}^{(s)}$ is a $\Sigma_{\mathbf{q},\bs\theta}$-decomposition, and $\Lambda\subset\{1,\ldots,r\}$ is a subset such that $\sum_{s\in\Lambda}\mathbf{d}^{(s)}<l\bs\delta$ is a root, then there is some $s'\notin\Lambda$ such that $\sum_{s\in\Lambda\sqcup\{s'\}}\mathbf{d}^{(s)}\le l\bs\delta$ and is a root. An induction then shows that we can refine $\mathbf{d}=\sum_{s=1}^r\mathbf{d}^{(s)}$ into $l\bs\delta+\cdots+l\bs\delta$.
\end{proof}
\begin{Thm}\label{Thm-dec-char}
Suppose that $\mathcal{M}_{\bs\theta}(\mathbf{q},\mathbf{d})$ is nonempty. Then,
\begin{itemize}
\item[(i)] $\mathbf{d}$ admits a minimal $\Sigma_{\mathbf{q},\bs\theta}$-decomposition $\mathbf{d}=\sum_{t=1}^r\mathbf{d}^{(t)}$.
\end{itemize}
Write $\mathbf{d}=\sum_{t\in\Lambda}m_t\mathbf{d}^{(t)}$ so that the vectors $\mathbf{d}^{(t)}$ are distinct for $t\in\Lambda\subset\{1,2,\ldots,r\}$. Then, 
\begin{itemize}
\item[(ii)] There is an isomorphism
$$
\prod_{t\in\Lambda}S^{m_t}\mathcal{M}_{\bs\theta}(\mathbf{q},\mathbf{d}^{(t)})\lisom\mathcal{M}_{\bs\theta}(\mathbf{q},\mathbf{d}).
$$
\end{itemize}
\end{Thm}
\begin{proof}
The statement is parallel to \cite{CB02}, and some arguments already appeared in the proof of \cite[Theorem 6.17]{ST}. The two parts (i) and (ii) are simultaneously proved by an induction on $|\mathbf{d}|:=\sum_{v\in Q_0}d_v$ via the following steps:
\begin{itemize}
\item[(1)] If there is a vertex $v$ with $(\mathbf{d},e_v)>0$ and either $q_v\ne 1$ or $\theta_v\ne 0$, then the reflection at $v$ reduces the problem to a $\mathbf{d}'$ with $|\mathbf{d}'|<|\mathbf{d}|$. 
\item[(2)] If there is a vertex $v$ with $(\mathbf{d},e_v)>0$, $q_v=1$ and $\theta_v=0$, then precisely the same arguments for \cite[Lemma 5.1]{CBS} show that $e_v$ must appear as a composition factor of $\mathbf{d}$, and thus the problem is reduced to $\mathbf{d}-e_v$.
\item[(3)] If $(\mathbf{d},e_v)\le0$ for every vertex $v$, we may pass to the connected components of the support quiver of $\mathbf{d}$ and assume that $\mathbf{d}$ lies in the fundamental region. By \cite[Theorem 6.16]{ST}, the problem can be further reduced to the following situations.
\item[(4)] $\mathbf{d}\in\Sigma_{\mathbf{q},\bs\theta}$, in which case the statements (i) and (ii) are trivial.
\item[(5)] $\mathbf{d}=m'l\bs\delta$ is of type $(\mathbf{Aff})$ and $\mathbf{q}^{\bs\delta}$ has order $l$. In this case, Lemma \ref{Lem-min-admi-m'ldelta} and Theorem \ref{Thm-Sym-iso} prove what we need.
\item[(6)] The support of $\mathbf{d}$ is $J\sqcup K$ and the only arrow connecting $J$ and $K$ is $a:\infty\rightarrow 0$ with $\infty\in J_0$ and $0\in K_0$; moreover, $d_0=d_{\infty}=1$ and $\mathbf{q}^{\mathbf{d}|_J}=1$. The proof of \cite[Theorem 6.17]{ST} reduces the problem to $\mathbf{d}|_J$ and $\mathbf{d}|_K$.
\item[(7)] $\mathbf{d}=e_{\infty}+m\bs\delta$ is of type $(\mathbf{Aff}^{\infty})$. We need to show that every $\bs\theta$-polystable $\mathbf{d}$-dimensional representation $\rho$ decomposes as $\rho_{\infty}\oplus\rho_1$, where $\rho_{\infty}$ is the simple representation corresponding to the simple root $e_{\infty}$. Then, the problem is reduced to step (5), thus completing the proof of the theorem.
\end{itemize}

Step (7) is part of the statement \cite[Theorem 6.17 (iii)]{ST}, but it seems that a proof is not provided there. Besides, the proof of \cite[Proposition 7.2]{ST} asserts that it follows from the arguments of \cite[\S 5]{CB02} verbatim, but this assertion does not seem to be true. Indeed, the proof in \textit{op. cit.} relies on \cite[\S 4]{CB02}, which in turn relies on a choice of total ordering on the filed $\mathbb{C}$; however, in the multiplicative setting, it is hard to imagine a meaningful total ordering on $\mathbb{C}^{\ast}$. We give a more geometric proof below.

Consider the quiver with vertex set $Q\sqcup\{\infty\}$ as in ($\mathbf{Aff}^{\infty}$), where $Q$ is of affine Dynkin type. Taking direct sums of representations induces a morphism
$$
\Psi:\mathcal{M}_{\bs\theta|_Q}(\mathbf{q}|_Q,\mathbf{d}|_Q)\times\mathcal{M}(q_{\infty},d_{\infty})\longrightarrow\mathcal{M}_{\bs\theta}(\mathbf{q},\mathbf{d}).
$$
Recall that $q_{\infty}=d_{\infty}=1$ and $\mathbf{d}|_Q=m\bs\delta$ for some $m\ge2$, and note that $\mathcal{M}(q_{\infty},d_{\infty})$ is a point. Let us show that $\Psi$ is an isomorphism, which is equivalent to our claim that every $\rho\in\mathcal{M}_{\bs\theta}(\mathbf{q},\mathbf{d})$ decomposes as $\rho_{\infty}\oplus\rho_1$. Consider this morphism at the level of representation spaces. Denote by $\mu_{\infty}$ the quasi-Hamiltonian moment map associated to the quiver of type ($\mathbf{Aff}^{\infty}$) as in (\ref{eq-quasi-Ham}) and by $\mu$ the map associated to $Q$. Then, taking the direct sum with $\rho_{\infty}$ defines a closed immersion $\mu^{-1}(\mathbf{q}|_Q)\hookrightarrow\mu_{\infty}^{-1}(\mathbf{q})$. It follows that $\Psi$ is a closed immersion. Since $\mathbf{d}$ is indivisible, Proposition \ref{Prop-Conn-nongen} shows that $\mathcal{M}_{\bs\theta}(\mathbf{q},\mathbf{d})$ is connected, and thus irreducible. However, both $\mathcal{M}_{\bs\theta}(\mathbf{q}|_Q,\mathbf{d}|_Q)$ and $\mathcal{M}_{\bs\theta}(\mathbf{q},\mathbf{d})$ have dimension $2m$; therefore, $\Psi$ is an isomorphism.
\end{proof}

\addtocontents{toc}{\protect\setcounter{tocdepth}{-1}}
\bibliographystyle{alpha}
\bibliography{BIB}
\end{document}